\documentclass[11pt,reqno]{article}

\usepackage{amsmath,amsthm,amsfonts,amssymb,latexsym,mathrsfs,url,enumerate}
\usepackage{graphicx,ifpdf,epstopdf}
\usepackage{comment}
\usepackage{amsmath}
\usepackage{dsfont}
\usepackage[colorlinks=true,
linkcolor=blue,citecolor=blue,
urlcolor=blue]{hyperref}

\newtheorem{theorem}{Theorem}
\theoremstyle{plain}

\newtheorem{claim}[theorem]{Claim}

\newtheorem{conjecture}[theorem]{Conjecture}
\newtheorem{corollary}[theorem]{Corollary}

\newtheorem{definition}[theorem]{Definition}

\newtheorem{lemma}[theorem]{Lemma}

\numberwithin{equation}{section}
\numberwithin{theorem}{section}
\def\F{\mathcal{F}}

\def\h{\mathcal{H}}

\def \e{\epsilon}

\begin{document}
\title{Hamiltonian paths and cycles in some 4-uniform hypergraphs}
\author{Guanwu Liu\thanks{Partially supported by Anhui Initiative in Quantum Information Technologies grant (No. AHY150200).
E-mail address: liuguanwu@hotmail.com}\\
School of Mathematical Sciences\\ University of Science and Technology of China\\
 Hefei 230026, P.R. China\\
and \\
School of Mathematics\\  Georgia Institute of Technology\\ Atlanta, GA 30332-0160, USA\\
\bigskip \\
 Xiaonan Liu\thanks{Corresponding author. Email address: xliu729@gatech.edu; Partially supported by NSF DMS-1856645 and NSF DMS-1954134} \\
School of Mathematics\\
Georgia Institute of Technology\\
Atlanta, GA 30332-0160, USA}

\date{}

\maketitle

\begin{abstract}
In 1999, Katona and Kierstead conjectured that if a $k$-uniform hypergraph $\h$ on $n$ vertices
has minimum co-degree $\lfloor \frac{n-k+3}{2}\rfloor$, i.e., each set of $k-1$ vertices
is contained in at least $\lfloor \frac{n-k+3}{2}\rfloor$ edges, then it has a Hamiltonian
cycle. R\"{o}dl, Ruci\'{n}ski and Szemer\'{e}di in 2011
proved that the conjecture is true when $k=3$ and $n$ is large. We show that this Katona-Kierstead
conjecture holds if $k=4$, $n$ is large, and $V(\h)$ has a partition $A$, $B$ such that
$|A|=\lceil n/2\rceil$, $|\{e\in E(\h):|e \cap A|=2\}| <\e n^4$ for a fixed small constant $\e>0$.
\end{abstract}

\section{Introduction}
A classical result of Dirac \cite{Dirac1952} states
that any graph on $n$ vertices with minimum degree at least $n/2$ contains a Hamiltonian cycle, and
$K_{\lceil \frac{n}{2} \rceil-1, \lfloor \frac{n}{2} \rfloor +1}$ shows that this is best possible.
However, paths and cycles may be defined in several ways for hypergraphs \cite{Bermond1976,Gy2008,Haxell2006,
Katona1999, Kuhn2006}.

A hypergraph is called $k$-uniform if every edge of it contains $k$ vertices.
For $k$-uniform hypergraphs (or $k$-graphs, for short) with $k\geq 3$, we consider
{\it paths} which are $k$-graphs with vertices $v_1,v_2, \cdots, v_l$ and edges
$\{v_i, v_{i+1},\cdots, v_{i+k-1}\}$, $i=1,\cdots,l-k+1$.
A {\it cycle} is defined similarly with the additional edges
$\{v_i, v_{i+1},\cdots, v_{i+k-1}\}$ for $i=l-k+2,\cdots,l$, where for $h\geq l$ we set $v_h=v_{h-l}$.
A {\it Hamiltonian path (cycle)} in a $k$-graph $\h$ is a path (cycle) which is a sub-hypergraph of
$\h$ and contains all vertices of $\h$.

Given a $k$-graph $\h$ and $T\in \binom{V(\h)}{k-1}$, the {\it neighbourhood} of $T$ is denoted by
$N_{\h}(T):=\{x: T\cup \{x\}\in E(\h)\}$. The {\it collective degree} (or {\it co-degree}, for short)
of $T$ is  $d_{k-1}(T):=|N_\h(T)|$.
The {\it minimum co-degree} of $\h$ is $\delta_{k-1}(\h):=\min\{ d_{k-1}(T):T\in \binom{V(\h)}{k-1}\}$.

Katona and Kierstead \cite{Katona1999} proved that if $\h$ is an $n$-vertex $k$-graph with
$\delta_{k-1}(\h)\geq (1-\frac{1}{2k})n-k+4$,
then $\h$ contains a Hamiltonian cycle. In the same paper, they make the following conjecture.

\begin{conjecture}\label{conj1}(Katona and Kierstead \cite{Katona1999})
Let $\h$ be a $k$-graph on $n\geq k+1\geq 4$ vertices. If $\delta_{k-1}(\h)\geq
\lfloor\frac{n-k+3}{2}\rfloor$, then $\h$ has a Hamiltonian cycle.
\end{conjecture}

The bound on $\delta_{k-1}(\h)$ is best possible due to a construction of a non-Hamiltonian $k$-graph on $n$
vertices with $\delta_{k-1}(\h)=\lfloor\frac{n-k+3}{2}\rfloor-1$. We describe the constuction
for $k=4$. Let $\h_0:=\h_0(A,B)$ be a 4-graph with vertex set $V=A\cup B$ with $A\cap B=\emptyset$,
$|A|=\lceil n/2\rceil$  and $|B|=\lfloor n/2 \rfloor$. Its edge set consists  of all
$\binom{|A|}{3}|B|+|A|\binom{|B|}{3}$ quadruples of vertices having an odd intersection with A.
It is easy to see that if $|A|, |B|\geq 2$ then $\delta_3(\h_0)=\lfloor n/2 \rfloor-2=
\lceil\frac{n-1}{2}\rceil-2$ and $\h_0$ does not have a Hamiltonian path.
In \cite{Rodl2011}, R\"{o}dl, Ruci\'{n}ski and Szemer\'{e}di
prove that Conjecture \ref{conj1} is true when $k=3$ and $n$ is large.

\begin{theorem}(R\"{o}dl, Ruci\'{n}ski and Szemer\'{e}di \cite{Rodl2011})
Let $\h$ be a 3-graph on $n$ vertices, where $n$ is sufficiently large. If $\delta_2(\h)\geq \lfloor n/2\rfloor$,
then $\h$ has a Hamiltonian cycle. Moreover, for every $n$ there exists an $n$-vertex 3-graph $\h_n$ such
that $\delta_2(\h_n)= \lfloor n/2\rfloor-1$ and $\h_n$ does not have a Hamiltonian cycle.
\end{theorem}

For a 4-graph $\h$ on $n$ vertices, let $A$, $B$ be a partition of $V(\h)$ and $\h(A,A,B,B):
= \{ e\in E(\h):|e\cap A|=2 \}$, and let $b(\h):=\min|\h(A,A,B,B)|$,
where the minimum is taken over all partitions $V(\h)=A\cup B$ with $|A|=\lceil n/2\rceil$
and $|B|=\lfloor n/2 \rfloor$. We know that if $b(\h)$ is very small,
then $\h$ is very ``close'' to the $\h_0$, see Claim \ref{edges} below. We show
that Conjecture~\ref{conj1} holds for these $\h$ with small $b(\h)$.

\begin{theorem}\label{main}
There exists $\e_0>0$ such that, for sufficiently large $n$ and any 4-graph $\h$ on $n$ vertices with
$b(\h)<\e_0 n^4$, the following hold:
\begin{enumerate}
\item[\textup{(i)}]If $\delta_3(\h)\geq \lceil\frac{n-1}{2}\rceil-1$, then $\h$ has a Hamiltonian path;
\item[\textup{(ii)}]If $\delta_3(\h)\geq \lfloor\frac{n-1}{2}\rfloor$, then $\h$ has a Hamiltonian cycle.
\end{enumerate}
\end{theorem}

The bound in (i) is tight because of $\h_0$. The bound in (ii)
is tight because of $\h'_0$, where $\h'_0$ is obtained from $\h_0$ by adding a new vertex $v$ and
joining it to all $\binom{n}{3}$ triples of vertices. We can see that (i) is a corollary of (ii).
Indeed, for $n$ even the thresholds in (i) and (ii) coincide. For $n$ odd, however, they differ by 1. Suppose $\h$ is a $4$-graph satisfying the conditions in  (i). In order to see the implication in this case, consider a 4-graph $\h'$ obtained from $\h$ by adding a new vertex $v$ and join it to all $\binom{n}{3}$ triples of vertices. Then
$$\delta_3(\h')\geq \delta_3(\h)+1\geq (\lceil\frac{n-1}{2}\rceil-1)+1\geq \lfloor\frac{n}{2}\rfloor= \lfloor\frac{(n+1)-1}{2}\rfloor$$
and by (ii) $\h'$ has a Hamiltonian cycle. After removing $v$, $\h$ has a Hamiltonian path.
We do not determine the optimal value of the
constant $\e_0$ in the theorem. We only checked that $\e_0=10^{-20}$ is sufficient.

For convenience, we will consider only the case when $\h$ has an even number of vertices. The odd case can be treated by some easy modifications and it is discussed in Section 5.

The rest of the paper is organized as follows. In Section 2, we study the typicality of vertices and edges of $\h$ as in \cite{Rodl2011}. The proofs of (i) and (ii) in Theorem \ref{main} will be given in Sections 3 and 4, respectively. Although (i) is a corollary of (ii), the proof of (i) given here better illustrates the proof approach of both results without involving too much technicality. Hence we also provide the proof of (i) here. In the final section, we offer some concluding remarks.

\section{The typicality of vertices and edges of $\h$}
Throughout this section, unless there are special instructions, $\h_0$ denotes the 4-graph with
$V(\h_0)=A\cup B$, where $A\cap B=\emptyset$ and $|A|=|B|$, and $E(\h_0)$ consisting of all quadruples
of $V(\h_0)$ each of
which intersects $A$ in precisely one or three vertices. For a 4-graph $\h$ with $V(\h)=V(\h_0)$,
we use notation $\h(A,B)$ and $\h_0(A,B)$ to indicate the partition.
We will refer to the edges with exactly three vertices in $A$ as the $AAAB$ edges, the edges with
exactly one vertex in $A$ as the $ABBB$ edges, etc. The $AAAB$ edges and the
$ABBB$ edges will be referred to as the {\it typical} $edges$ of $\h$, and the $AABB$ edges will be
called ${\it atypical}$. (The $AAAA$ edges and $BBBB$ edges remain ${\it neutral}$.)

First we show the following claim which says that if $b(\h)$ is small and $\delta_3(\h)$ is large,
then $\h$ almost contains a copy of $\h_0$.

\begin{claim}\label{edges}
Suppose $\h$ is a 4-graph with $V(\h)=A\cup B$, such that $A\cap B=\emptyset$ and $|A|=|B|=n$.
For any $c,c_1>0$, if $|\h(A,A,B,B)|<cn^4$ and $\delta_3(\h)\geq (1-c_1)n$, then
\begin{equation*} %\label{0}
| E(\h_0(A,B))\backslash  E(\h)|\leq \frac{1}{3}(c_1+4c)n^4+O(n^3).
\end{equation*}
\end{claim}

\begin{proof}
For convenience,
% let $|ABBB|$, $|AABB|$ and $|AAAB|$ denote the number of edge set with exactly one,
%two and threevertices in $A$ of $\h$,
let $ABB$ and $AAB$ denote the sets of $3$-vertex subset of $V(\h)$ %three vertex subset of $V(\h)$
with exactly one and two vertices
from $A$ respectively. Then
\begin{equation*}
\sum_{S\in ABB}d_3(S)=2|AABB|+3|ABBB|\geq (1-c_1)n\cdot n\cdot \binom{n}{2}
\end{equation*}
and
\begin{equation*}
\sum_{S\in AAB}d_3(S)=2|AABB|+3|AAAB|\geq (1-c_1)n\cdot n\cdot \binom{n}{2}.
\end{equation*}
Summing the above two equations, we have
\begin{equation*}
3|ABBB|+3|AAAB|\geq 2(1-c_1)n\cdot n\cdot \binom{n}{2}-4|AABB|.
\end{equation*}
Since the number of edges of $\h_0(A,B)$ is $n\cdot\binom{n}{3}+\binom{n}{3}\cdot n$ and $|AABB|< cn^4$,
we have
\begin{equation*}
|E(\h_0(A,B))\setminus E(\h)|\leq \frac{1}{3}(c_1+4c)n^4+O(n^3).
\end{equation*}
\end{proof}

From time to time, we also need to deal with hypergraphs whose vertex partitions are not balanced. Therefore, in the remainder of this section we always assume that $\h$ is a 4-graph on $2n$ vertices and $A, B$ is a partition of $V(\h)$
such that
\begin{equation}\label{eq1}
\delta_3(\h)\geq n-1,
\end{equation}

\begin{equation}\label{eq2}
n-5\e_0 n\leq |A|\leq n+5\e_0 n,
\end{equation}
and
\begin{equation}\label{eq3}
|\h(A,A,B,B)|\leq \e_0 n^4,
\end{equation}
where $\e_0>0$ is sufficiently small and $n$ is sufficiently large.

\subsection{Classification of vertices}
We follow the notation and the set up in \cite{Rodl2011}.
The {\it link} of a vertex $v\in V(\h)$ is defined as the set of triples $L_v:=\{uwt:uwtv\in E(\h)\}$;
let $L_v^{V_1V_2V_3}:=L_v\cap V_1V_2V_3$ and $l_v^{V_1V_2V_3}:=|L_v^{V_1V_2V_3}|$,
where $V_1V_2V_3\in\{AAA,AAB,ABB,BBB\}$. Similarly,
the link of a pair $u,v\in V(\h)$ is defined as the set of pairs $L_{uv}:=\{wt:uvwt\in E(\h)\}$;
let $L_{uv}^{V_1 V_2}:=L_{uv}\cap V_1 V_2$ and $l_{uv}^{V_1V_2}:=|L_{uv}^{V_1V_2}|$,
where $V_1V_2\in \{AA,AB,BB\}$.

%Let $a$ be a vertex in $A$ and $b$ be a vertex in $B$.
 In the remainder of this section, vertices $a$ and  $a_i$ (respectively, $b$ and $b_i$) are contained in $A$
(respectively, $B$).
 From \eqref{eq1}, we see that
 \begin{equation}\label{eq4}
2l_{a}^{AAB}+2l_{a}^{ABB}\geq |B|(|A|-1)(n-1) \ \ and \ \ 6l_{a}^{BBB}+2l_{a}^{ABB}\geq |B|(|B|-1)(n-1);
\end{equation}
and
\begin{equation}\label{eq5}
2l_{b}^{AAB}+2l_{b}^{ABB}\geq |A|(|B|-1)(n-1) \ \ and \ \ 6l_{b}^{AAA}+2l_{b}^{AAB}\geq |A|(|A|-1)(n-1).
\end{equation}

The vertices of $\h$ are classified according to the values of $l_{v}^{ABB}$ and $l_{v}^{AAB}$ as follows:

\begin{definition}
For $\e>0$ and vertex $a\in A$, $a$ is called
\begin{itemize}
\item $\e$-typical if $l_{a}^{ABB}\leq \e |A| \binom{|B|}{2}$;
\item $\e$-medium if $l_{a}^{ABB}>\e |A| \binom{|B|}{2}$ and $l_{a}^{AAB}>\e \binom{|A|}{2} |B|$;
\item an $\e$-anarchist if $l_{a}^{AAB}\leq \e \binom{|A|}{2} |B|$.
\end{itemize}

\medskip

Similarly, for vertex $b\in B$, $b$ is called
\begin{itemize}
\item $\e$-typical if $l_{b}^{AAB}\leq \e \binom{|A|}{2} |B|$;
\item $\e$-medium if $l_{b}^{AAB}>\e \binom{|A|}{2} |B|$ and $l_{b}^{ABB}>\e |A| \binom{|B|}{2}$;
\item an $\e$-anarchist if $l_{b}^{ABB}\leq \e |A| \binom{|B|}{2}$.
\end{itemize}
\end{definition}

We have the following observations:

{\bf Observation} (i) For clarity, results and proofs below are presented in the balanced case,
when $|A|=|B|=n$, but they remain valid, except for Claim~\ref{fact1},
in non-balanced case with just slightly worse constants.

{\bf Observation} (ii) By \eqref{eq4} and \eqref{eq5},  if $a\in A$ is $\e$-typical then
\begin{equation}\label{onevert1}
l_{a}^{AAB}\geq \frac{1}{2}n(n-1)^2-\frac{1}{2}\e n^3 \ \ and \ \ l_{a}^{BBB}\geq \frac{1}{6}n(n-1)^2-\frac{1}{6}\e n^3;
\end{equation}
and if $b\in B$ is $\e$-typical then
\begin{equation}\label{onevert2}
l_{b}^{ABB}\geq \frac{1}{2}n(n-1)^2-\frac{1}{2}\e n^3 \ \ and \ \ l_{b}^{AAA}\geq \frac{1}{6}n(n-1)^2-\frac{1}{6}\e n^3.
\end{equation}
Hence  each vertex of $\h$ only belongs to one of the above three types when $n$ is sufficiently large.

{\bf Observation} (iii) Assume \eqref{eq2} holds. For sufficiently large $n$, if $a\in A$ is an
$\e$-anarchist, let $A'=A\setminus \{a\}$ and $B'=B\cup \{a\}$. If \eqref{eq2} still holds for $A'$, $B'$,
then $l_{a}^{A'A'B'}=l_{a}^{AAB}\leq \e|B| \binom{|A|}{2} \leq \e' |B'|\binom{|A'|}{2}$ for some
$\e'>\e$.  For any other vertex $v\neq a$, $l_{v} ^{V_1V_2V_3}$ is changed by no more than
$\max \{\binom{|A|}{2}, \binom{|B|}{2}, |A||B|\}=O(n^2)$.
If $v$ is $\e$-typical with respect to $A,B$, say $v\in A$, then $l_v^{ABB}\leq \e |A|\binom{|B|}{2}$
and $l_v^{A'B'B'}\leq l_v^{ABB}+O(n^2)< \e' |A'|\binom{|B'|}{2}$.
So, transferring an $\e$-anarchist $a$ in $A$ to $B$ makes $a$ $\e'$-typical with respect to $A',B'$,
and other $\e$-typical vertices with respect to $A, B$  are $\e'$-typical with respect to $A', B'$.

\medskip

%The following claim tells us that all ``anarchists" with respect to $A,B$, where $b(\h)=|\h(A,A,B,B)|$,
%are contained in only one class of $A,B$. So if the number of  such ``anarchists "in $\h$ is small,
%we can apply Observation (iii) to get more ``typical" vertices.
By Observation (iii), we know that an anarchist acts like a typical vertex on the other side.
We claim that in the case of a balanced partition $(A,B)$ such that $|\h(A,A,B,B)|=b(\h)$,
coexistence of an anarchist with an atypical vertex on the other side is impossible.

\begin{claim}\label{fact1}
Suppose $|A|=|B|=n$ and $b(\h)=|\h(A,A,B,B)|$. For every $\e>0$ and sufficiently large $n$, if there is an $\e$-anarchist in $B$
then every vertex in $A$ is $3\e$-typical. Also, if there is an
$\e$-anarchist in $A$ then every vertex in $B$ is $3\e$-typical.
\end{claim}

\begin{proof}
For $v\in V$, define $I_v=l_v^{AAB}-l_v^{ABB}$. Then, for $a\in A$,
\[
I_a=l_a^{AAB}-l_a^{ABB}=|\h(A\setminus \{a\},A\setminus \{a\},B\cup \{a\}, B\cup \{a\})|-|\h(A,A,B,B)|,
\]
while for $b\in B$,
\[
I_b=l_b^{AAB}-l_b^{ABB}=|\h(A,A,B,B)|-|\h(A\cup \{b\},A\cup \{b\},B\setminus \{b\}, B\setminus \{b\})|.
\]
Thus, for all $a\in A$ and $b\in B$,
\[
|\h(A\setminus \{a\}\cup \{b\},A\setminus \{a\}\cup \{b\},B\setminus \{b\}\cup \{a\},
B\setminus \{b\}\cup \{a\})|=|\h(A,A,B,B)|+I_a-I_b+O(n^2).
\]
Here the $O(n^2)$ term comes from the edges $abuv$, where $uv\in N_{\h}(a,b)$.
Hence, by the minimality of $b(\h)$, we must have
\[
I_a \geq I_b-O(n^2).
\]

Suppose that there exists $a\in A$ and $b\in B$ such that $l_{b}^{ABB}\leq \frac{\e}{2} n^3$
and $l_{a}^{ABB}>\frac{3}{2}\e n^3$. Then by \eqref{eq5},
\[
I_b=l_b^{AAB}-l_b^{ABB}=l_b^{AAB}+l_b^{ABB}-2l_b^{ABB}\geq \frac{1}{2}n^3-\e n^3
\]
and
\[
I_a=l_a^{AAB}-l_a^{ABB}< \frac{1}{2}n^3-\frac{3}{2}\e n^3\leq I_b-\frac{1}{2}\e n^3,
\]
a contradiction.

The proof of the second statement is analogous.
\end{proof}

The next claim justifies the name ``typical" and it shows that the number of atypical
vertices is small.

\begin{claim}\label{claim1}
Assuming \eqref{eq1}, \eqref{eq2} and \eqref{eq3}, for all $\e_0, \e_1>0$, less than $8(\e_0/\e_1)n$
vertices in $\h$ are $\e_1$-atypical. Among them, less than $5\e_0 n$ vertices in $A$ and less than
$5\e_0 n$ vertices in $B$ are $\e_1$-anarchists, provided $\e_1<1/5$.
\end{claim}

\begin{proof}
Let $x$ be the number of $\e_1$-atypical vertices in $\h$. Then, since each of these vertices contributes
more than $\frac{1}{2}\e_1 n^3$ edges to $|\h(A,A,B,B)|$, and every such
edge is counted at most four times, we have
\[
\frac{1}{4}x \cdot \frac{1}{2}\e_1 n^3<\e_0 n^4,
\]
which implies that $x<8(\e_0/\e_1)n$.

Now, let $x'$ be the number of $\e_1$-anarchists in A. By \eqref{eq4}, every $\e_1$-anarchist
$a\in A$ contributes at least $l_a^{ABB}\geq \frac{1}{2} |B| (|A|-1|)(n-1)-l_a^{AAB}\geq
\frac{1}{2}(1-\e_1)n^3-O(n^2)$ edges to $|\h(A,A,B,B)|$, and these edges are counted at most twice. Hence
\[
\frac{1}{2}x'\cdot (\frac{1}{2}(1-\e_1) n^3-O(n^2))<\e_0 n^4,
\]
which implies $x'<5\e_0 n$ since $\e_1<1/5$.

The proof of the statement $b\in B$ is analogous.
\end{proof}

\medskip

Now we classify the pair of vertices in $\h$ by the values $l_{uv}^{AA}$, $l_{uv}^{AB}$ or $l_{uv}^{BB}$
as follows.
\begin{definition}
Fix $\e>0$. A pair of vertices
\begin{itemize}
\item $\{a_1,a_2\}$ is $\e$-typical if $l_{a_1 a_2}^{BB}\leq \e \binom{|B|}{2}$;
\item $\{a,b\}$ is $\e$-typical if $l_{a b}^{AB}\leq \e |A||B|$;
\item $\{b_1,b_2\}$ is $\e$-typical if $l_{b_1 b_2}^{AA}\leq \e \binom{|A|}{2}$;
\item $\{u,v\}\subseteq V(\h) $ is $(\e_1,\e_2)$-typical if both $u$ and $v$ are $\e_1$-typical
and the pair $\{u,v\}$ is $\e_2$-typical.
\end{itemize}
\end{definition}

%Assuming that (1),(2),(3), the hyperedges containing a typical pair in $\h$ are almost contained in $\h_0$.
{\bf Observation.} From \eqref{eq1},
 \begin{align*}
 &l_{a_1a_2} ^{AB}+2l_{a_1a_2}^{BB}\geq |B| (n-1) ,\\&
 l_{ab}^{AB}+2l_{ab}^{AA}\geq (|A|-1)(n-1)  \ and\ l_{ab}^{AB}+2l_{ab}^{BB}\geq (|B|-1)(n-1),\\&
 l_{b_1b_2}^{AB}+2l_{b_1b_2}^{AA}\geq |A|(n-1).
 \end{align*}

Hence, if $\{ a_1, a_2\}$, $\{a,b\}$ and $\{b_1,b_2\}$ are $\e$-typical, then by definition,
we have
\begin{equation}\label {pair1}
l_{a_1 a_2}^{AB} \geq n(n-1)- \e n^2,
\end{equation}
\begin{equation}\label{pair2}
l_{a b}^{AA}+l_{ab}^{BB} \geq (n-1)^2-\e n^2,
\end{equation}
\begin{equation}\label{pair3}
l_{b_1b_2}^{AB}\geq n(n-1)-\e n^2.
\end{equation}

\medskip

Next we show that each typical vertex is contained in a small number of atypical pairs.
\begin{claim}\label{claim2}
Assuming \eqref{eq1} and \eqref{eq2}, for all $\e_1,\e_2>0$, every $\e_1$-typical vertex in $A$
belongs to at most $(\e_1/\e_2)n$ $\e_2$-atypical pairs in $AA$ and at most $(\e_1/\e_2)n$ $\e_2$-atypical
pairs in $AB$. Moreover, every $\e_1$-typical vertex in $B$ belongs to at most
$(\e_1/\e_2)n$ $\e_2$-atypical pairs in $BB$ and at most $(\e_1/\e_2)n$ $\e_2$-atypical pairs in $AB$.
\end{claim}

\begin{proof}
Let $a\in A$ be $\e_1$-typical. If $a$ belongs to more than $(\e_1/\e_2)n$ $\e_2$-atypical pairs in $AA $, then
 \[
l_{a}^{ABB}>\frac{\e_2}{2}n^2\times \frac{\e_1}{\e_2}n=\frac{\e_1}{2}n^3 ,
\]
contradicting the $\e_1$-typicality of $a$. Similarly, if $a$ belongs to more than
$(\e_1/\e_2)n$ $\e_2$-atypical pairs in $AB $, then,
\[
 l_{a}^{ABB}>\frac{1}{2}\e_2 n^2\times \frac{\e_1}{\e_2}n=\frac{\e_1}{2}n^3,
\]
a contradiction.

The proof of the statement for $\e_1$-typical vertex in $B$ is analogous.
\end{proof}

The triples of vertices in $\h$ are classified as follows.
\begin{definition}
Fix $\e>0$. A triple of vertices
\begin{itemize}
\item $\{a_1,a_2,a_3\}$ is $\e$-typical if $d_B(a_1,a_2,a_3)\geq (1-\e)|B|$;
\item $\{a_1,a_2,b\}$ is $\e$-typical if $d_B(a_1,a_2,b)\leq \e|B|$;
\item $\{a,b_1,b_2\}$ is $\e$-typical if $d_A(a,b_1,b_2)\leq \e|A|$;
\item $\{b_1,b_2,b_3\}$ is $\e$-typical if $d_A(b_1,b_2,b_3)\geq (1-\e)|A|$;
\item $\{u,v,w\}\subseteq V(\h)$ is $(\e_1,\e_2,\e_3)$-typical if each of $u$, $v$ and $w$ is
$\e_1$-typical, each of pairs $\{u,v\}$, $\{v,w\}$ and $\{u,w\}$ is $\e_2$-typical, and the triple
 $\{u,v,w\}$ is $\e_3$-typical.
\end{itemize}
\end{definition}

 %We show that the hyperedges containing a typical triple in $\h$ are almost contained in $\h_0$.
{\bf Observation.}
From \eqref{eq1},
 \begin{align*}
d_A(a_1,a_2,b)+d_B(a_1,a_2,b)\geq n-1, \\
d_A(a,b_1,b_2)+d_B(a,b_1,b_2)\geq n-1.
 \end{align*}

Hence, if $\{ a_1, a_2,a_3\}, \{a,b_1,b_2\}, \{a_1,a_2,b\}$ and $\{b_1,b_2,b_3\}$ are $\e$-typical,
then by definition, we have
\begin{equation}\label {tp1}
d_B(a_1,a_2,a_3)\geq n-1-\e n,
\end{equation}
\begin{equation}\label{tp2}
d_A(a_1, a_2,b) \geq n-1-\e n,
\end{equation}
\begin{equation}\label{tp3}
d_B(a, b_1,b_2) \geq n-1-\e n,
\end{equation}
\begin{equation}\label {tp4}
d_A(b_1,b_2,b_3)\geq n-1-\e n.
\end{equation}

The following two claims show that any typical vertex or typical pair is contained
in a small number of atypical triples.
\begin{claim}\label{claim3}
Assuming \eqref{eq1} and \eqref{eq2}, for all $\e_1,\e_3>0$, every $\e_1$-typical vertex in $A$ belongs to
at most $(\e_1/\e_3)n^2$ $\e_3$-atypical triples in each type of
$AAB, ABB$ and $AAA$. Moreover, every $\e_1$-typical vertex in $B$ belongs to at most
$(\e_1/\e_3)n^2$ $\e_3$-atypical triples in each type of
$AAB, ABB$ and $BBB$.
\end{claim}

\begin{proof}
Let $a\in A$ be $\e_1$-typical. If $a$ belongs to more than $(\e_1/\e_3)n^2$ $\e_3$-atypical triples in
$AAB$ or more than $(\e_1/\e_3)n^2$ $\e_3$-atypical triples in
$ABB$, then
\[
l_{a}^{ABB}>\frac{1}{2}(\e_1/\e_3)n^2\times \e_3 n=\frac{\e_1}{2}n^3 \ \ or \ \ l_{a}^{ABB}>
(\e_1/\e_3) n^2\times \e_3 n>\frac{\e_1}{2}n^3,
\]
contradicting the $\e_1$-typicality of $a$. Let $x$ be the number of $\e_3$-atypical triples in $AAA$.
Then by (2.4),
\[
(\frac{1}{2}-\frac{\e_1}{2})n^3\leq l_{a}^{AAB}=\sum_{a_1,a_2\neq a}d_B(a_1,a_2,a)\leq x(1-\e_3)n
+(\frac{1}{2}n^2-x)n=\frac{1}{2}n^3-x\e_3 n.
\]
So $x\leq (\e_1/2\e_3)n^2\leq (\e_1/\e_3)n^2$.

The proof of the statement for $\e_1$-typical vertex in $B$ is analogous.
\end{proof}

\begin{claim}\label{claim4}
Assuming \eqref{eq1} and \eqref{eq2}, for all $\e_2,\e_3>0$, every $\e_2$-typical pair $\{a_1,a_2\}$, or
$\{a,b\}$, or $\{b_1,b_2\}$ belongs to at most $(\e_2/\e_3)n$ $\e_3$-atypical
triples in each of the four types $AAA$, $AAB$, $ABB$, and $BBB$.
\end{claim}

\begin{proof}
Let $\{a_1,a_2\}$ be an $\e_2$-typical pair. If $\{a_1,a_2\}$ belongs to more than
$(\e_2/\e_3)n$ $\e_3$-atypical triples in $AAB$, then
\[
l_{a_1 a_2}^{BB}>\frac{1}{2}(\e_2/\e_3)n\times \e_3 n=\frac{\e_2}{2}n^2,
\]
contradicting the $\e_2$-typicality of $\{a_1,a_2\}$.
Let $x$ be the number of $\e_3$-atypical triples in $AAA$. Since
$l_{a_1a_2}^{AB}+2l_{a_1 a_2}^{BB} \geq |B|(n-1)$, we have
\[
(1-\e_2)n^2\leq l_{a_1a_2}^{AB}=\sum_{a\neq a_1,a_2}d_B(a_1,a_2,a)\leq x(1-\e_3)n+(n-x)n=n^2-x\e_3 n;
\]
we have $x\leq (\e_2/\e_3)n$.

The proof of the statement for $\{a,b\}$ and $\{b_1,b_2\}$ are analogous.
\end{proof}

If $|\h(A,A,B,B)|$ is small, then by Claim~\ref{claim1}, $\h$ does not contain too many atypical vertices.
Next, we claim that the number of atypical triples in $\h$ is also small.
%By the above claims, we derive the following corollary stating that $\h$ contains a small number
%of ``atypical" vertices.
\begin{corollary}\label{coro1}
Assuming \eqref{eq1}, \eqref{eq2} and \eqref{eq3}, for all $\e_0,\e_1,\e_2,\e_3>0$ and for every $\e_4\geq 16(\e_0/\e_1)+4(\e_1/\e_2)+(\e_1/\e_3)$, every set of at least $\e_4n^3$ triples in $\binom{V(\h)}{3}$ contains at least one $(\e_1,\e_2,\e_3)$-typical triple. In particular, there are less than $\e_4n^3$ triples in $\binom{V(\h)}{3}$ which are not $(\e_1,\e_2,\e_3)$-typical.
\end{corollary}

\begin{proof}
It suffices to count all triples $\{u,v,w\}\subseteq V(\h)$, such that at least one of them is $\e_1$-atypical,
or all of $\{u,v,w\}$
are $\e_1$-typical and one of the pairs from $\{u,v,w\}$ is not $\e_2$-typical, or all vertices
are $\e_1$-typical and all pairs from $\{u,v,w\}$ are $\e_2$-typical, but $\{u,v,w\}$ is not $\e_3$-typical.

By Claim~\ref{claim1}, the number of triples, of which at least one vertex is $\e_1$-atypical, is at most
$8(\e_0/\e_1) n \times \binom{2n-1}{2}$.
By Claim~\ref{claim2}, the number of triples, of which all three vertices are $\e_1$-typical but at least
one pair is $\e_2$-atypical, is at most $2n \times (2(\e_1/\e_2) n)\times (2n-2)\times \frac{1}{2!}$.
By Claim~\ref{claim3}, the number of triples, of which all vertices  are $\e_1$-typical and all pairs
are $\e_2$-typical but $\{u,v,w\}$ is not $\e_3$-typical, is at most $2n\times (3(\e_1/\e_3) n^2 )\times \frac{1}{3!}$.

Hence, the number of  all these atypical triples are at most $\e_4 n^3$.
\end{proof}

\subsection{Short paths between typical triples}
In this section, we prove that if $|H(A,A,B,B)|$ is small and $\delta_3(\h)$ is large then certain
typical triples can be connected by a path of length at most 12.
Recall that the 4-graph $\h_0=\h_0(A,B)$
consists of all $AAAB$ and $ABBB$ quadruples. (Here, we allow non-balanced partitions $(A, B)$; however,
they must satisfy \eqref{eq2}.) A sextuple of vertices $(v_1,v_2,v_3,$ $w_1,w_2,w_3)$ is called {\it $\h_0$-connected} if both $\{v_1,v_2,v_3\}$ and $\{w_1,w_2,w_3\}$ belong to $AAV$ or both $\{v_1,v_2,v_3\}$ and $\{w_1,w_2,w_3\}$ belong to $BBV$. We can call it an $\h_0$-connected sextuple formed by the triples $\{v_1,v_2,v_3\}$ and $\{w_1,w_2,w_3\}$.
%and there exists a path in $\h_0$ connecting these two triples.
Given a set of vertices $K$, a path $P$ is $K$-${\it avoiding}$ if $V(P)\cap K=\emptyset$.
A subset of vertices $T\subseteq V(\h)$ is said to be $\h_0$-${\it complete}$ if $E(\h[T])\supseteq E(\h_0[T])$.
We show that for an $\h_0$-connected sextuple formed by two $(\e_1,\e_2,\e_3)$-typical triples, there is a path
in $\h_0$ connecting these two triples.

%We now prove Claim~\ref{triconn} implying that for any two ``typical" triples belonging to $AAV$ or $BBV$,
%the sextuple consisting of these two triples is $\h_0$-connected and the path connecting them has length at most 11.
\begin{claim}\label{triconn}
Let $\e_0,\e_1,\e_2,\e_3$ be sufficiently small and assume that \eqref{eq1}, \eqref{eq2} and \eqref{eq3} hold.
Let $(v_1,v_2,v_3,w_1,w_2,w_3)$ be an $\h_0$-connected sextuple in $\h$, where $\{v_1,v_2,v_3\}$ and $\{w_1,w_2,w_3\}$ are two $(\e_1,\e_2,\e_3)$-typical triples. For every set of vertices $K\subseteq V(\h)\setminus \{v_1,v_2,v_3,w_1,w_2,$ $w_3 \}$ with $|K|\leq \frac{2}{3}n$, there exists a subset $T\subseteq V(\h)\setminus (K\cup \{v_1,v_2,v_3,w_1,w_2,w_3 \} )$ such that $|T\cap A|,|T\cap B|\geq 5$, and $T\cup \{v_1,v_2,v_3 \}$ and $T\cup \{w_1,w_2,w_3 \}$ are $\h_0$-complete. In particular, there exists a $K$-avoiding path $P$ in $\h$ with at most 12 vertices such that the end triples of $P$ are $\{v_1,v_2,v_3 \}$ and $\{w_1,w_2,w_3 \}$ and all edges in $P$ are typical.
%moreover, if both triples are in $AAV$ (or $BBV$), all edges in $P$ are $AAAB$ (or $ABBB$) edges.
\end{claim}

\begin{proof}
We select a set $T$ at random, by choosing each vertex of $V(\h)\setminus (K\cup \{v_1,v_2,v_3,w_1,w_2,w_3 \})$
independently with probability $p=60/n$. We will show that it satisfies all required properties
with positive probability.

Let $E_v$ and $E_w$ be the events that the subsets $T\cup \{v_1,v_2,v_3 \}$ and
$T\cup \{w_1,w_2,w_3 \}$ are not $\h_0$-complete, and let $E=E_v\cup E_w$. We claim that
\[
\mathds{P}(E_v)\leq P_0+3P_1+3P_2+P_3,
\]
where $P_0$ is the probability that $T$ is not $\h_0$-complete, $P_1$ is the probability that there exist
$x,y,z\in T$ such that $v_i x y z\in E(\h_0) \setminus E(\h[T\cup \{v_i \}])$,
$P_2$ is the probability that there exist $x,y\in T$ such that $v_i v_j x y\in E(\h_0) \setminus
E(\h[T\cup \{v_i,v_j \}])$, and $P_3$ is the probability that there exist $x\in T$ such that
$v_1 v_2 v_3 x\in E(\h_0)\setminus E(\h[T\cup \{v_1,v_2,v_3 \}])$.

By Claim \ref{edges} with $c=\epsilon_0$ and $c_1=1/n$,
we know $|E(\h_0)\setminus E(\h)| \leq 2\e_0 n^4$.
(Although the partition of $V(\h)$ might not be balanced, the result of Claim \ref{edges} still holds with a larger constant.)
Thus, $P_0 \leq 2\e_0 n^4 p^4$.
By~\eqref{onevert1} and~\eqref{onevert2}, for any $1\leq i\leq 3$, the number of edges of $\h_0$ containing
$v_i$ that are not edges of $\h$ is at most $ \e_1 n^3$, since $v_i$ is $\e_1$-typical. Thus, $P_1\leq  \e_1 n^3p^3$.
By~\eqref{pair1},~\eqref{pair2} and~\eqref{pair3}, for any $1\leq i \neq j\leq 3$, the number of edges in $\h_0$
containing the pair $\{v_i,v_j\}$  that are not edges of $\h$ is at most $\e_2 n^2$, since $\{v_i,v_j\}$ is
$\e_2$-typical. Thus, $P_2 \leq \e_2 n^2 p^2$.
By~\eqref{tp1},~\eqref{tp2},~\eqref{tp3} and~\eqref{tp4}, the number of edges in $\h_0$ containing the triple
$\{v_1, v_2, v_3\}$ that are not edges of $\h$ is at most $\e_3 n$, since $\{v_1,v_2,v_3\}$ is $\e_3$-typical.
Thus, $P_3\leq \e_3 np$.

Hence,
\[
P(E_v)\leq 2\e_0n^4p^4+3\cdot \e_1 n^3p^3+3\cdot \e_2 n^2p^2+\e_3 np<\frac{1}{4}
\]
for $\e_0,\e_1,\e_2,\e_3$ sufficiently small.  Similarly, $P(E_w)<\frac{1}{4}$.

Finally, recalling that $
|A\setminus (K\cup \{v_1,v_2,v_3,w_1,w_2,w_3 \})|\geq \frac{1}{3}n-6>\frac{1}{4}n+4,
$
we have
\[
P(|T\cap A|\leq 4)\leq (1+np+\binom{n}{2}p^2+\binom{n}{3}p^3+\binom{n}{4}p^4)(1-p)^{\frac{n}{4}}<\frac{1}{4}.
\] Similarly, $P(|T\cap B|\leq 4)<1/4$. Hence, the required set $T$ does exist.

\medskip

Consider the case when an $H_0$-connected sextuple is formed by two $(\e_1,\e_2,\e_3)$-typical triples $\{a_1, a_2, a_3\}$ and $\{a_4, a_5, a_6\}$. By the above argument, the required set $T$ exists. Suppose $\{b_1, a, a', a'', b_2\}\subseteq T$. Then by the properties of $T$, $P=a_1a_2a_3 b_1 a a' a'' b_2 a_4a_5a_6$ is a $K$-avoiding path with $11$ vertices in $\h$ and all edges of $P$ are $AAAB$ edges. For other cases, it can be checked that the two $(\e_1,\e_2,\e_3)$-typical triples in any $H_0$-connected sextuple can be connected by a $K$-avoiding path with at most $12$ vertices in which every edge is typical. Moreover, if both triples are in $AAV$ (or $BBV$), all edges in this $K$-avoiding path connecting these two triples are $AAAB$ (or $ABBB$) edges.
\end{proof}
%Note that if an $\h_0$-connected sextuple of vertex $\{v_1,v_2,v_3,w_1,w_2,w_3 \}$ satisfying the conditions of Claim \ref{triconn}, then there exists a path $P$ in $V(\h)\backslash K$ with at most 12 vertices and whose end triplescare $\{v_1,v_2,v_3 \}$ and $\{w_1,w_2,w_3 \}$. For example,  consider an $H_0$-connected sextuple formed by two $(\e_1,\e_2,\e_3)$-typical triples $\{a_1, a_2, a_3\}$ and $\{a_4, a_5, a_6\}$. By Claim~\ref{triconn}, there exits $T\supseteq \{b_1, a, a', a'', b_2\}$ such that $a_1a_2a_3 b_1 a a' a'' b_2 a_4a_5a_6$ is a path in $\h$. It can be checked that the two $(\e_1,\e_2,\e_3)$-typical triples in any $H_0$-connected sextuple can be connected by a path in which every edge is typical. Moreover, if both triples are in $AAV$ (or $BBV$), all edges in the path  obtained by Claim~\ref{triconn} connecting these two triples are $AAAB$ (or $ABBB$) edges.
\section{Hamiltonian paths}
In this section, we prove the following

\begin{theorem}\label{E}
There exists $\e_0>0$ such that, for sufficiently large $n$ and any $4$-graph $\h$ on $2n$ vertices with
$b(\h)<\e_0 n^4$ the following holds.
If $\delta_3(\h)\geq n-1$, then $\h$ has a Hamiltonian path.
\end{theorem}

For typical vertices, we want to use paths similar to these in $\h_0$ to connect them. So we need to deal
with atypical vertices, which are medium vertices or anarchists. By Claim~\ref{claim1}, the number of such
vertices is small. In our proof, we find a path to absorb all medium vertices.  By Claim~\ref{fact1}, the
anarchists can only exist on one side. We may transfer all anarchists to the other side, so that all vertices will be typical in a new partition.

First, we introduce a structure called {\it bridge}, which helps us construct a path containing all
medium vertices.

\begin{definition}
Given $\e_1,\e_2,\e_3>0$, an $(\e_1,\e_2,\e_3)$-bridge is a path of at most $800$ vertices
whose end triples are $(\e_1,\e_2,\e_3)$-typical with one in $AAA$ and the other in $BBB$.
\end{definition}

For convenience, for some small $\e$ we set
\[
\e_0=\e^4, \ \ \e_1=\e^3, \ \ \e_2=\e^2, \ \ \e_3=\e, \ \ \e_4=40\e, \ \ \e_5=120\e.
\]

 The proof of Theorem \ref{E} can be described in four steps:
 Build a bridge $M$ (cf. Lemma \ref{le1});
 arrest all medium vertices by a path $Q$ containing $M$ (cf. Lemma \ref{le2});
 transfer all anarchists not belonging to $Q$ to the other side of the partition;
 complete the Hamiltonian path $P$ (cf. Lemma \ref{le3}).

\subsection{Building a bridge}

\begin{lemma}\label{le1}
For sufficiently small $\e>0$, $\h$ contains an $(\e_1,\e_2,\e_3)$-bridge $M$ with at most $25$ vertices.
\end{lemma}
\begin{proof}
%[Proof of Lemma~\ref{le1}]

Fix two $(\e_1,\e_2)$-typical pairs $\{a_1,a_2\}$ and $\{b_1,b_2\}$.
Suppose $a_1 a_2 b_1 b_2\in E(\h)$. Since $\delta_3(\h)\geq n-1$ and $\{a_1, a_2\}, \{b_1,b_2\}$ are
$(\e_1,\e_2)$-typical, it follows from Claim~\ref{claim4} that there exists $x\in N(a_1,a_2,b_1)\setminus\{b_2\}$ and $y\in N(a_1,b_1,b_2)\setminus\{a_2,x\}$, such that $\{x, a_1,a_2\}$ and $\{b_1,b_2,y\}$ are $(\e_1, \e_2, \e_3)$-typical. Hence, $x a_2 a_1 b_1 b_2 y$ is a path in $\h$. Now we show that the path $P=x a_2 a_1 b_1 b_2 y$ can be extended to an $(\e_1,\e_2,\e_3)$-bridge by Claim~\ref{triconn}. By Corollary~\ref{coro1}, there exist $(\e_1,\e_2,\e_3)$-triples $\{a_1',a_2',a_3'\}$ and $\{b_1',b_2',b_3'\}$ disjoint from $V(P)$. Since $\{a_1',a_2',a_3'\}$ and $\{x,a_1,a_2\}$ are $AAV$ triples, there exists a $\{b_1,b_2,y, b_1',b_2',b_3'\}$-avoiding path $P_1=a_1'a_2'a_3'\cdots x a_1 a_2$ with at most $12$ vertices by Claim~\ref{triconn}. Similarly, there exists a $V(P_1)$-avoiding path $P_2=b_1b_2y\cdots b_1'b_2'b_3'$ with at most $12$ vertices. Hence, we extend $P$ to $P_1\cup P \cup P_2$ such that $P_1\cup P \cup P_2$ is an $(\e_1,\e_2,\e_3)$-bridge with at most $24$ vertices and the end triples are $\{a_1',a_2',a_3'\}$ and $\{b_1',b_2',b_3'\}$.

So assume $a_1 a_2 b_1 b_2\notin E(\h)$. Let $X=N(a_1,a_2,b_1)$ and $Y=N(a_1,b_1,b_2)$.
Since $a_1 a_2 b_1 b_2\notin E(\h)$, we have $X\cup Y\subseteq V(\h)\setminus \{a_1,a_2,b_1,b_2\}$
and $|X\cup Y|\leq 2n-4$. Since $\delta_3(\h)\geq n-1$, $|X|\geq n-1$ and $|Y|\geq n-1$.
So $|X\cap Y|=|X|+|Y|-|X\cup Y|\geq 2$ implies that there exists a vertex $z\in X\cap Y$. Similarly, by
Claim~\ref{triconn}, we can find $x\in N(a_1,a_2,z)$ and $y\in N(b_1,b_2, z)$, such that $x a_2 a_1 z b_1 b_2 y$
can be extended to the desired bridge. Hence, in any case, we can always find an $(\e_1,\e_2,\e_3)$-bridge in $\h$ with no moth than 25 vertices.
\end{proof}

The construction in the proof of Lemma \ref{le1} can also be used for 3-graphs, which would shorten
Section 8 in \cite{Rodl2011}.

\subsection{Taking care of atypical vertices}
First, we need a simple claim from \cite{Rodl2011}.

\begin{claim}\label{claim5}(R\"{o}dl, Ruci\'{n}ski and Szemer\'{e}di \cite{Rodl2011})
Given $a>0$ and $k\geq 2$, every $k$-graph $\F$ with $m$ vertices and with at least $a\binom{m}{k}$ edges
contains a path on at least $am/k$ vertices.
\end{claim}

\begin{lemma}\label{le2}
Let $z_1,\ldots,z_{t_1}$ be the $\e_5$-medium vertices and $K\subseteq V(\h)$ with $|K|< \e^3 n$. There exist pairwise disjoint $K$-avoiding paths $Q_1, \ldots, Q_{t_1}$ such that for every integer $i$ such that $1\le i \le t_1$, all edges in $Q_i$ are typical, and $Q_i$ contains $z_i$ with $|V(Q_i)|=7$ and both end triples $(\e_1,\e_2,\e_3)$-typical.
%moreover, if $z_i\in A$ (repectively, $B$), then all edges in $Q_i$ are $AAAB$ (respectively, $ABBB$) edges and both end triples belong to the type $AAB$ (respectively, $ABB$).
In particular, assume that $M$ is an $(\e_1,\e_2,\e_3)$-bridge. Then there exists a path $Q$ of length at most $\e^3 n$, which contains $M$ and all $\e_5$-medium vertices of $\h$, and whose end triples, one in $AAA$ and one in $BBB$, are $(\e_1,\e_2,\e_3)$-typical. Moreover, all edges in $Q-M$ are typical.
%there are two paths in $P-M$; one path contains all medium vertices in $A\setminus V(M)$ and all edges of this path are $AAAB$ edges while the other contains all medium vertices in $B\setminus V(M)$ and all edges are $ABBB$ edges.
\end{lemma}

\begin{proof}
%Let $z_1,\ldots,z_{t_1}$ be the $\e_5$-medium vertices. By Claim \ref{claim1}, we know that $t_1\leq 8(\e_0/\e_5)n$. Suppose that we have already found paths $Q_j$ for $j=1,\ldots,i-1$, where $Q_j$ is a path containing $z_j$ with $|V(Q_j)|=7$ and both end triples $(\e_1,\e_2,\e_3)$-typical. Moreover, if $z_j\in A$ (repectively, $B$), then both end triples of $Q_j$ belong to the type $AAB$ (respectively, $ABB$). We do a induction on the number of such paths. Set $z=z_i$.
By Claim \ref{claim1}, we know that $t_1\leq 8(\e_0/\e_5)n$. We do an induction on the number of such paths. Suppose that we have already found paths $Q_j$ for $j=1,\ldots,i-1$, such that $Q_j$ satisfies the properties in Lemma~\ref{le2}. Set $z=z_i$.

We may assume $z\in A$ as the proof is analogous for $z\in B$.
Let $G_z^{AAB}$ be the set of $(\e_1,\e_2,\e_3)$-typical triples in $L_z^{AAB}$.
By Corollary~\ref{coro1}, $l_z^{AAB}-|G_z^{AAB}|\leq \e_4 n^3$. Then $|G_z^{AAB}|\geq l_z^{AAB}-
\e_4 n^3\geq \frac{\e_5}{2}n^3-\e_4 n^3$, since $z$ is $\e_5$-medium. Further, let $F_z^{AAB}=G_z^{AAB}
[(V(\h)\setminus K)\setminus U_i]$, where $U_i=\bigcup_{j=1}^{i-1}V(Q_j)$. Note that
$|F_z^{AAB}|\geq |G_z^{AAB}|-(|U_i|+|K|) n^2$ and $(|U_i|+|K|)n^2\leq
(7(i-1)+\frac{1}{2} \e^3 n)n^2\leq 7t_1 n^2 + \e^3 n^3< 2\e^3 n^3$ for sufficiently large $n$.
Thus, by the above estimates and because $z$ is $\e_5$-medium, we have
\[
|F_z^{AAB}|\geq |G_z^{AAB}|-2\e^3 n^3 \geq (\frac{\e_5}{2}n^3-\e_4 n^3)-
2\e^3 n^3\ >10\e n^3;
\]
so by Claim~\ref{claim5}, $F_z^{AAB}$ contains a path of length six, i.e.,
$a_1 a_2 b_1 a_3 a_4 b_2$. Then $Q_i=a_1 a_2 b_1 z a_3 a_4 b_2$, disjoint from $Q_1,\ldots, Q_{i-1}$,
gives the desired path, since $Q_i$ is $K$-avoiding, the end triple of $Q_i$ are $(\e_1,\e_2,\e_3)$-typical $AAB$ triples and all edges in $Q_i$ are $AAAB$ edges.

\medskip
Now we have a given bridge $M$. Let $K=V(M)$ as $|V(M)|\le 800\le \e^3 n$. Set $Q_{z_i}=Q_i$ for $i=1, \ldots, t_1$.
By Claim~\ref{triconn}, for all $\e_5$-medium vertices $w\in A\setminus V(M)$, since the end triples of $Q_w$ are in $AAB$ and $(\e_1,\e_2,\e_3)$-typical, we connect all paths $Q_w$ into a $V(M)$-avoiding path, denoted by $Q_{top}$.
Similarly, for all $\e_5$-medium vertices $w\in B\setminus V(M)$, we connect all paths $Q_w$ into a $V(M\cup Q_{top})$-avoiding path, denoted by $Q_{zig}$, since the end triples of $Q_w$ are in $BBA$ and $(\e_1,\e_2,\e_3)$-typical.
Now we use the given bridge $M$ to connect $Q_{top}$ and $Q_{zig}$. We connect the end triple of $M$ in $AAA$ with one $AAB$ end triple of $Q_{top}$, and connect the end triple of $M$ in $BBB$ with one $BBA$ end triple of $Q_{zig}$, also by Claim~\ref{triconn}. Then we obtain a path $P$, which contains $M$ and all $\e_5$-medium vertices of $\h$, whose end triples are $(\e_1,\e_2,\e_3)$-typical, and one of its triples is in $AAB$ and the other is in $BBA$.

By Claim~\ref{triconn}, $|V(P)|\leq 8(\e_0/\e_5) n \cdot (7+(12-6))+ |V(M)|+2\cdot (12-6)$. By Corollary~\ref{coro1}, there exists an $(\e_1,\e_2,\e_3)$-typical triple $\{a_0,a_1,a_2\}$ in $AAA$ and an $(\e_1,\e_2, \e_3)$-typical triple $\{b_0,b_1,b_2\}$ in $BBB$, such that $\{a_0,a_1,a_2,b_0,b_1,b_2\}\cap V(P)=\emptyset$. We apply Claim~\ref{triconn} to connect  $\{a_0,a_1,a_2\}$ with the $AAB$ end triple of $P$, and connect $\{b_0,b_1,b_2\}$ with the $BBA$ end triple of $P$. This gives a path $Q=a_2 a_1 a_0  \cdots  b_0  b_1 b_2$, containing $M$ and all $\e_5$-medium vertices of $\h$, such that $|V(Q)|\leq |V(P)|+2\cdot(12-3)\leq 8(\e_0/\e_5)n \cdot (7+6)+ |V(M)|+2\cdot 6 + 2\cdot 9\leq \e^3 n$.
\end{proof}

\subsection{Completing the Hamiltonian path}
To complete the proof of Theorem~\ref{E}, we need the following lemma in \cite{Reiher2019}.
\begin{lemma} [Reiher, R\"{o}dl, Ruci\'{n}ski, Schacht and Szemer\'{e}di \cite{Reiher2019}]\label{3gphc}
Every 3-graph with $n$ vertices and minimum vertex degree at least
$(\frac{5}{9}+o(1))\binom{n}{2}$ has a Hamiltonian cycle.
\end{lemma}
Lemma~\ref{le2} gives a path $Q$ containing all $\e_5$-medium vertices. By Claim~\ref{fact1},
we know that if there exists an $\e_5$-anarchist in one side of the vertex partition, then
all vertices in the other side are $3\e_5$-typical. Moreover, the number of $\e_5$-anarchists
is less than $5\e_0 n$. So we transfer all such vertices to the other side of the vertex partition.
Then all vertices in $V(\h)\setminus V(Q)$ are $4\e_5$-typical with respect to the new partition.
We use Lemma~\ref{3gphc} to derive the following.

\begin{lemma}\label{seq}
Assume $\eqref{eq1}$, $\eqref{eq2}$ and \eqref{eq3} hold. Let $X$ be a set of $4\e_5$-typical vertices with
$m=|X|\geq cn$, where $c$ is a constant. Suppose $\{x_0,x_1,x_2\}$ and $\{x_0',x_1',x_2'\}$ are two disjoint $
(4\e_5,\e_5^{3/4},\e_5^{1/2})$-typical triples disjoint from $X$. For sufficiently large $n$
and sufficiently small $\e$, there exists a sequence of vertices
$x_0x_1x_2x_3\cdots x_{m+2} x_{m+3} (=x_2') x_{m+4} (=x_1') x_{m+5} (=x_0'),$
such that all $\{x_i,x_{i+1},x_{i+2}\}$ are $(4\e_5,\e_5^{3/4},\e_5^{1/2})$-typical
for $0\leq i \leq m+3$ and $X=\{x_3, x_4,$ $\ldots, x_{m+2}\}$.
\end{lemma}
\begin{proof}

Construct a 3-graph $G_X$ with vertex set $V(G_X)=X\cup \{x_2\}$ and edge set
\begin{align*}
E(G_X)= & \{ x_2 u v: \{u,v\}\subseteq X \text{ such that $\{x_2, u,v\}, \{x_1,x_2,u\}, \{x_1',x_2',v\}$
 and $\{u,v, x_2'\}$} \\
& \text{are $\e_5^{1/2}$-typical \} } \cup \{ uvw: \{ u,v,w\} \subseteq
X \text{ is $(4\e_5,\e_5^{3/4},\e_5^{1/2})$-typical} \}.
\end{align*}
%\begin{align*}
%E(G_X)= & \{ x_2 uv: \{u,v\}\subseteq X  \text {\ such that $\{x_2, u,v\}$ and $\{x_1,x_2,u\}$
%are $\e_5^{1/2}$-typical} \} \cup \\
%&  \{ x'_2 uv: \{u,v\}\subseteq X  \text {\ such that $\{x'_2, u,v\}$ and $\{x'_1,x'_2,u\}$
%are $\e_5^{1/2}$-typical} \} \cup \\
%&  \{ uvw: \{ u,v,w\} \subseteq X \text{ is $(4\e_5,\e_5^{3/4},\e_5^{1/2})$-typical} \}.
%\end{align*}
%We show that $\delta_1(G_X)>\frac{2}{3} \binom{m}{2}$. Since the pair $\{x_1,x_2\}$ is $\e_5^{3/4}$-typical,
%each of them belongs to at most $2\e_5^{1/4}n$ $\e_5^{1/2}$-atypical triples in $X$ by Claim~\ref{claim4}.
%Thus, there are at least $m-\e_5^{1/4}n$ vertices $u$ such that $\{x_1,x_2,u\}$ is
%$\e_5^{1/2}$-typical. Since $x_2$ is $4\e_5$-typical, by Claim~\ref{claim3}, the number of $\e_5^{1/2}$-atypical
%triples in $X$ containing $x_2$ is at most $(4\e_5^{1/2})n^2$. Hence,
%$d_{G_X}(x_2)\geq \binom {m- \e_5^{1/4} n} {2} -(8\e_5^{1/2})n^2 >\frac{2}{3} \binom{m}{2}$ because $m\geq cn$.
%Similarly, we obtain that $d_{G_X}(x'_2) >\frac{2}{3} \binom{m}{2}$.
We show that $\delta_1(G_X)>\frac{2}{3} \binom{m}{2}$. Since the pairs $\{x_1,x_2\}$  and $\{x_1',x_2'\}$
are $\e_5^{3/4}$-typical, each of them belongs to at most $2\e_5^{1/4}n$ $\e_5^{1/2}$-atypical triples in
$X$ (by Claim~\ref{claim4}); so there are at least $\binom{m-4\e_5^{1/4}n}{2}$ pairs of $\{u,v\}$ such that
$\{x_1,x_2,u\}, \{x_1',x_2',v\}$ are $\e_5^{1/2}$-typical. The number of $\e_5^{1/2}$-atypical triples
containing  $x_2$ or $x_2'$ is at most $2\cdot 3\cdot 4\e_5^{1/2}n^2$ (by Claim~\ref{claim3} as they are
$4\e_5$-typical). Thus, $d_{G_X}(x_2)\geq \binom{m-4\e_5^{1/4}n}{2}-24\e_5^{1/2}n^2>\frac{2}{3}\binom{m}{2}$
because $m=|X|\geq cn$. Since all vertices in $X$ are $4\e_5$-typical, by Claim~\ref{claim2} and \ref{claim3},
the number of $\e_5^{3/4}$-atypical
pairs in $X$ containing a fixed $4\e_5$-typical vertex is at most $2\cdot 4\e_5^{1/4} n$, and the number of
$\e_5^{1/2}$-atypical triples in $X$ containing a fixed $4\e_5$-typical vertex is at most
$3\cdot 4\e_5^{1/2} n^2$. Thus, for any vertex $u\in X$, we have
$d_{G_X}(u)\geq \binom{m}{2} - 16 \e_5^{1/4} m n-12\e_5^{1/2} n^2> \frac{2}{3} \binom{m}{2}$.

Since $|V(G_X)|=m+1$ and $m\geq cn$, by Lemma~\ref{3gphc}, $G_X$ has a Hamiltonian cycle. So we
can find a Hamiltonian path in $G_X$, say $P_X=x_2 x_3
\cdots x_{m+2}$, such that $\{x_1,x_2,x_3\}, \{x_{m+1},x_{m+2},x_2'\}$ and $\{x_{m+2}, x_2', x_1'\}$ are
$\e_5^{1/2}$-typical. Hence, we obtain a sequence of vertices
$$x_0 x_1x_2 x_3 x_4\cdots x_{m+2} x_{m+3} ( =x_{2}') x_{m+4} (=x_1') x_{x_1+5} ( =x_0'),$$
where $\{x_i,x_{i+1},x_{i+2}\}$
is $\e_5^{1/2}$-typical for $ 0\leq i\leq m+3$ and $X=\{x_3,x_4,\ldots,x_{m+2}\}$.
\end{proof}

Now we are ready to prove the following lemma, which implies Theorem~\ref{E}.

\begin{lemma}\label{le3}
Suppose that $\h$ contains a path $Q=a_2a_1a_0\cdots b_0b_1b_2$ of length at most $\e^3 n$ such that
\begin{itemize}
\item $\{a_0,a_1,a_2\}\in AAA$ and $\{b_0,b_1,b_2\}\in BBB$, and both are $(2\e_1,2\e_2,2\e_3)$-typical;
\item every vertex of $V(\h)\setminus V(Q)$ is $4\e_5$-typical.
\end{itemize}
Then $Q$ can be extended to a Hamiltonian path in $\h$.
\end{lemma}

\begin{proof}
We use typical edges to
connect all remaining vertices in $\h$. Note that all vertices in $V(\h)\setminus V(Q)$ are $4\e_5$-typical,
but with respect to a (possibly) slightly modified partition, still denoted by $(A,B)$,
in which the two sides may differ in size by at most $10\e_0 n$.
Let $A'=A\setminus V(Q)$ and $B'=B\setminus V(Q)$ and let $m_1=|A'|$ and $m_2=|B'|$.
Without loss of generality, suppose $m_1\leq m_2$.
Then $m_1\geq n-5\e_0 n-\e^3 n\geq n-2\e^3 n$ and $m_2-m_1 \leq 2\e^3 n$.

First,  we label the vertices in $B'$. Since all vertices in $B'$ are $4\e_5$-typical and
$|B'|=m_2\geq n-2\e^3 n$, there exists a $(4\e_5,\e_5^{3/4},\e_5^{1/2})$-typical triple
$\{b_{m_2},b_{m_2+1},b_{m_2+2}\}$ such that  $b_{m_2},b_{m_2+1},b_{m_2+2}\in B'$. Applying
Lemma~\ref{seq} to $\{b_0,b_1,b_2\}$, $\{b_{m_2},b_{m_2+1},b_{m_2+2}\}$ and
$B'\setminus \{b_{m_2},b_{m_2+1},b_{m_2+2}\}$, we have a sequence, denoted by $P_B=b_0 b_1 \cdots b_{m_2+2}$
such that $\{b_i,b_{i+1},b_{i+2}\}$ is $\e_5^{1/2}$-typical for $0\leq i \leq m_2$ and $B'=\{b_3,b_4,\ldots,b_{m_2+2}\}$.
Construct an auxiliary bipartite graph $\Gamma_B$ between $A'$ and $B_1$,
where $B_1=\{b_2,b_5,\ldots,b_{3p_1-1}\}$ with $p_1=\lfloor\frac{m_2+3}{3}\rfloor$
such that for any  $a\in A'$ and $b_i\in B_1$, $ab_i\in E(\Gamma_B)$ if and only if $$a\in N(b_{i-2},b_{i-1},b_{i})\cap
N(b_{i-1},b_{i},b_{i+1})\cap N(b_{i},b_{i+1},b_{i+2})\cap N(b_{i+1},b_{i+2},b_{i+3}).$$

Observe that if we find a matching in $\Gamma_B$, then we find a path  in $\h$ similar to paths in $\h_0$.
Since $\{b_i,b_{i+1},b_{i+2}\}$ $(0\leq i \leq m_2) $ is typical, $d_{\Gamma_B}(b)\geq 0.99m_1$ for all
$b\in B'$; so at least $0.9m_1$ vertices $a\in A'$ have degree $d_{\Gamma_B}(a)\geq 0.9p_1$.
We need to deal with the vertices in $A'$ of small degree since vertices of larger degree can be included
in a matching of $\Gamma_B$.
Let $A_{big}=\{a\in A':d_{\Gamma_B}(a)\geq 0.9p_1\}$; so $|A_{big}|\geq 0.9m_1$. Let
$A_{small}=A'\backslash A_{big}$. We claim the following.

\begin{itemize}
\item[(1)] There exists a path $P_{top}$ in $\h$ such that $A_{small}\subseteq V(P_{top})
\cap A \subseteq \{a_0,a_1,a_2\}\cup A', V(P_{top})\cap B\subseteq B'$ and one end triple of $P_{top}$ is $a_0a_1a_2$.
%the set of previous $3t$ vertices of $P_A$ (|A_{small}|\leq 0.1 m_1\leq 2t)$.
\end{itemize}
Let $t:=\lceil\frac{3m_1-m_2}{8}\rceil$; then $\frac{3m_1-m_2-3}{8}\leq t \leq
\frac{3m_1-m_2+9}{8}$. Since $(1-2\e^3)n \leq m_1\leq m_2 \leq (1+\e^3) n$, we have
$0.24n\leq t \leq 0.25n$. Let $A_S$ be a subset of $A'$ such that $A_{small}\subseteq
A_S$ and $|A_S|=3t-3\geq 0.7n$. Note that $|A_{small}|\leq 0.1 m_1 \leq 3t-3$, so we can find such $A_S$. Since $|A'\setminus A_S|=m_1-(3t-3)\geq 0.2n$, there exist $a_{3t},a_{3t+1},a_{3t+2}\in A'\setminus A_S$ such that $\{a_{3t},a_{3t+1},a_{3t+2}\}$ is $(4\e_5,\e_5^{3/4},\e_5^{1/2})$-typical. We apply Lemma~\ref{seq} to $\{a_0,a_1,a_2\}$, $\{a_{3t},a_{3t+1},a_{3t+2}\}$ and $A_S$. Then there exists a sequence of vertices $a_0a_1a_2a_3\ldots a_{3t-1}a_{3t}a_{3t+1}a_{3t+2}$, such that $\{a_i,a_{i+1},a_{i+2}\}$ is $\e_5^{1/2}$-typical and $A_{small}\subseteq A_S=\{a_3,a_4,\ldots, a_{3t-1}\}$.

\begin{comment}
%We want to put some vertices in
%$(A \setminus V(Q))\cup \{a_0, a_1, a_2\}$ in order $a_0, a_1, a_2, a_3, \ldots, a_{3t+2},$
%such that $\{a_i, a_{i+1}, a_{i+2}\}$ is $\e_5^{1/2}$-typical for any $0\leq i\leq 3t$,
%and $A_{small}$ is included in $\{a_3, a_4, \dots, a_{3t-1}\}$. Note that $|A_{small}|\leq 0.1 m_1\leq 3t-3$.
Consider a 3-graph $G_{A_S}$ with vertex set $A_S\cup \{a_2\}$ and edge set
\begin{align*}
E(G_{A_S})=& \{a_2xy :\{x,y\}\subseteq {A_S}\text{ such that $\{a_2,x,y\}$ and
$\{a_1,a_2,x\}$ are $\e_5^{1/2}$-typical} \}\\
& \cup \{uvw:\{u,v,w\}\subseteq {A_S}  \text{ is $(4\e_5, \e_5^{3/4}, \e_5^{1/2})$-typical}\}
\end{align*}
Similarly, we can show that the minimum vertex degree of $G_{A_S}$ is at least $\frac{2}{3}
\binom{|A_S|}{2}$. By Lemma~\ref{3gphc}, $G_{A_S}$  has a Hamiltonian path, denoted by
$P_{A_S}=a_2a_3\ldots a_{3t-1}$, where $\{a_1,a_2,a_3\}$ is $\e_5^{1/2}$-typical and
$A_{small}\subseteq V(P_{A_S})$. Since the pair $\{a_{3t-2},a_{3t-1}\}$ is $\e_5^{3/4}$-typical,
there exist $a_{3t},a_{3t+1},a_{3t+2}\in A'\setminus A_S$ such that $\{a_{3t-2},a_{3t-1},a_{3t}\},
\{a_{3t-1},a_{3t},a_{3t+1}\}$ and $\{a_{3t},a_{3t+1},a_{3t+2}\}$ are $\e_5^{1/2}$-typical triples.
\end{comment}
%Recall that $t=\lceil\frac{3m_1-m_2}{8}\rceil$,
Construct another auxiliary bipartite graph $\Gamma_A$ between $A_2$ and $B''$, where
$A_2=\{a_2,a_5,\ldots,$ $a_{3t-1}\}$ and
$B''$ is the set of the last $p_2=\lceil 0.3 m_1\rceil$ vertices of $P_B$, i.e., $B''=\{b_{m_2+2},b_{m_2+1},
\ldots,$ $b_{m_2+3-p_2}\}$. For any $a_i\in A_2$ and $b\in B''$, $a_ib\in E(\Gamma_A )$ if and only if
$$ b\in N(a_{i-2},
a_{i-1},a_{i})\cap N(a_{i-1},a_{i},a_{i+1})\cap N(a_{i},a_{i+1},a_{i+2})\cap N(a_{i+1},a_{i+2},a_{i+3}).$$

Again, since $\{a_i,a_{i+1},a_{i+2}\} (0\leq i \leq t)$ is $\e_5^{1/2}$-typical, $d_{\Gamma_A}(a)\geq
0.99p_2$ for all $a\in A_2$; so at least $0.9p_2$ vertices $b\in B''$ have degree $d_{\Gamma_A}(b)\geq 0.9t$.
Let $B_{big}= \{ b\in B'': d_{\Gamma_A}(b)\geq 0.9 t\}$.
Then $|B_{big}|\geq 0.9 p_2>t$ (since $t \leq \frac{3m_1-m_2+9}{8}<0.26m_1<0.9\cdot\lceil 0.3 m_1 \rceil =0.9p_2$ for sufficiently large $n$).
We choose a set $\overline {B} \subseteq B_{big}$ of size $|\overline {B}|=t$. Consider the subgraph
$\Gamma_A'=\Gamma_A [A_2\cup \overline {B}]$.
Since $|A_2|=|\overline B|=t$ and all vertices in $A_2$ and $\overline B$ have degree at least $0.9t$ in $\Gamma_A'$,
there exists a perfect matching in $\Gamma_A'$ by Dirac's theorem. The perfect matching forms a path in $\h$,
denoted by $$P_{top}= a_0 a_1 a_2 \overline{b_1} a_3 a_4 a_5 \overline{ b_2} \cdots a_{3t-1} \overline{b_t}\cdots, $$
where $a_{3i-1}\overline{b_i}\in E(\Gamma_A)$. Note that
the end of $P_{top}$, possibly $a_{3t-3} a_{3t-2}a_{3t-1} \overline{b_t}$ or $a_{3t-2}a_{3t-1}\overline
{b_t}a_{3t}$ or $a_{3t-1}\overline{b_t}a_{3t}a_{3t+1}$ or $\overline {b_t}a_{3t}a_{3t+1}a_{3t+2}$, is determined
by the numbers $m_1$ and $m_2$. This proves (1).

We now show the following claim.
\begin{itemize}
\item[(2)] There exists a path $P_{zig}$ in $\h$ such that $V(P_{zig})\cap A= A'\setminus V(P_{top}),
V(P_{zig})\cap B= \{b_0,b_1,b_2\}\cup B'\setminus  V(P_{top})$, and one end triple of $P_{zig}$ is $b_0b_1b_2$.
\end{itemize}
Consider $P_{B_1}= b_0b_1b_2 \ldots b_{m_2+2-p_2}$ and $V(P_{B_1})=V(P_B)\setminus B''$.
Since all vertices in $B_2:=B''\setminus \overline B$ are $4\e_5$-typical and $|B_2|=p_2-t\geq 0.04n$, there
exist vertices $b_{m_2-t}',b_{m_2+1-t}',b_{m_2+2-t}'\in B_2$ such that $\{b_{m_2-t}',b_{m_2+1-t}',
b_{m_2+2-t}'\}$ is $(4\e_5,\e_5^{3/4},\e_5^{1/2})$-typical. We extend $P_{B_1}$ by applying Lemma~\ref{seq}
to $\{b_{m_2-p_2}, b_{m_2+1-p_2}, b_{m_2+2-p_2}\}$, $\{b_{m_2-t}'$, $b_{m_2+1-t}',b_{m_2+2-t}'\}$, and
$B_2 \setminus \{b_{m_2-t}',b_{m_2+1-t}',b_{m_2+2-t}'\}$.
%the pair $\{b_{m_2+1-p_2},b_{m_2+2-p_2}\}$ is $\e_5^{3/4}$-typical, we can use the vertices in
%$B''\setminus \overline B$ to extend $P_{B_1}$. Let $P_{B}'=b_0'
Hence, we obtain $b_0' b_1' b_2' \dots b_{m_2+2-t}'$ such that $b_j'=b_j$ for all $0\leq j\le m_2+2-p_2$,
$B_2=\{b_{m_2+3-p_2}',
\dots, b_{m_2+2-t}'\}$, and $\{b_i',b_{i+1}',b_{i+2}'\}$ is $\e_5^{1/2}$-typical for any $ 0\leq i\leq m_2-t$.

Similarly, construct an auxiliary bipartite graph $\Gamma_B'$ with partition classes $\overline {A}$, $B_3$, where
$\overline{A}=A\setminus V(Q\cup P_{top})$ and $B_3=\{b_2',b_5',\ldots,b_{3p_3-1}'\}$ with
$p_3=\lfloor\frac{m_2+3-t}{3}\rfloor$. We choose the end triple for $P_{top}$ to make $|\overline{A}|=|B_3|$.
For  any $a\in \overline{A}$ and $b_i'\in B_3$,
$ab_i'\in E(\Gamma_B')$  if and only if $$a\in N(b_{i-2}',b_{i-1}',b_{i}')\cap N(b_{i-1}',b_{i}',b_{i+1}')
\cap N(b_{i}',b_{i+1}',b_{i+2}')\cap N(b_{i+1}',b_{i+2}',b_{i+3}').$$
We know $d_{\Gamma_B'}(b)\geq 0.99 p_3$ for all $b\in B_3$. Since $A_{small}\subseteq V(P_{top})$, $\overline{A}\subseteq
A_{big}$ and $d_{\Gamma_B} (a)\geq 0.9 p_1$ for each $a\in \overline {A}$. Hence, $d_{\Gamma_B'} (a) \geq 0.9 p_1 - \lceil \frac{p_2}{3}
\rceil \geq 0.3m_2-0.1m_1\geq 0.8 p_3$. Therefore, $\delta(\Gamma_B')\geq 0.8 p_3$. By Dirac's theorem,
$\Gamma_B'$ has a perfect matching, which forms a path in $\h$, denoted by
$$P_{zig}= b_0' b_1' b_2' \overline {a_1} b_3' b_4' b_5' \overline {a_2} \cdots b_{3p_3-1}' \overline{a_{p_3}}\cdots,$$
where $\overline{a_i} b'_{3i-1}\in E(\Gamma_B')$.

\medskip

Now, $Q\cup P_{top} \cup P_{zig}$ is a Hamiltonian path in $\h$.
\end{proof}

\section{Hamiltonian cycles}
In this section, we prove the following
\begin{theorem}\label{F}
There exists $\e_0>0$ such that, for sufficiently large $n$ and any 4-graph $\h$ on $2n$ vertices with
$b(\h)<\e_0 n^4$ the following holds:
If $\delta_3(\h)\geq n-1$, then $\h$ has a Hamiltonian cycle.
\end{theorem}

The proof of Theorem \ref{F} proceeds along the lines of the proof of Theorem \ref{E}, except that to get a
Hamiltonian cycle we will need a second bridge. Suppose we have two disjoint bridges $M_1$ and $M_2$.
The arguments for taking care of medium and anarchist vertices are essentially the same as in the
proof of Theorem \ref{E}. So we have a path $Q$ which contains $M_1$ and all remaining medium vertices, and
is disjoint from $M_2$. Let $n_1=|A\setminus V(Q \cup M_2)|$, $n_2=|B\setminus V(Q \cup M_2)|$, and let $m=\frac{3n_1-n_2+6}{8}$.
When $m$ is an integer, we may apply the following variation of Lemma \ref{le3}, whose proof
basically repeats the proof of Lemma \ref{le3} with just minor modifications.

\begin{lemma}\label{le4}
Suppose that $\h$ contains a path $Q=a_2 a_1 a_0 \cdots b_0 b_1 b_2$ of length at most $\e^3 n$ and a
bridge $M_2=a'_2 a'_1 a'_0 \cdots b'_0 b'_1 b'_2$ such that
\begin{itemize}
\item $V(Q)\cap V(M_2)=\emptyset$;
\item $m=\frac{3n_1-n_2+6}{8} \ is \ an \ integer$;
\item all end triples $\{a_0,a_1,a_2\}$, $\{b_0,b_1,b_2\}$, $\{a'_0,a'_1,a'_2\}$ and $\{b'_0,b'_1,b'_2\}$
 are $(2\e_1,2\e_2,2\e_3)$-typical;
\item every vertex of $V(\h)\setminus V(Q\cup M_2)$ is $4\e_5$-typical.
\end{itemize}
Then $Q\cup M_2$ can be extended to a Hamiltonian cycle in $\h$. \end{lemma}

\begin{proof}
Let $A'=A\setminus V(Q\cup M_2)$ and $B'=B\setminus V(Q\cup M_2)$. It follows that $|A'|=n_1, |B'|=n_2$.
We build a top path $P_{top}$ with $V(P_{top})\cap A\subseteq \{a_0,a_1,a_2,a_2',a_1',a_0'\}\cup A'$
and $V(P_{top})\cap B\subseteq B'$ such that $|V(P_{top})\cap A|=3(m+1), |V(P_{top})\cap B|=m$, and $P_{top}$
connects two $AAA$ end triples of $Q$ and $M_2$, i.e., $a_0a_1a_2$ and $a_2' a_1'a_0'$.  Then, we use $P_{zig}$
with $V(P_{zig})\cap A=A'\setminus V(P_{top})$ and $V(P_{zig})\cap B=\{b_0,b_1,b_2,b_2',b_1',b_0'\}\cup B'
\setminus V(P_{top})$ to connect $b_0b_1b_2$ with $b_2'b_1'b_0'$.
%Let $n'_1=|A\setminus V(Q \cup M_2\cup P_{top})|=n_1-3(m-1)=m$ and $n'_2=|B\setminus V(Q \cup M_2\cup
%P_{top})|=n_2-m=3(m-1)$. Then, we use $P_{zig}$ with
Note that $|V(P_{zig})\cap A|=n_1-(3(m+1)-6)=m$ and $|V(P_{zig})\cap B|=6+n_2-m=3(m+1)$.
%to connect $b_0b_1b_2$ and $b_2' b_1' b_0'$.
Therefore, $Q\cup M_2 \cup P_{top} \cup P_{zig}$ is a Hamiltonian cycle in $\h$.

To construct $P_{top}$, we apply Lemma~\ref{seq} to $\{b_0,b_1,b_2\}, \{b_2',b_1',b_0'\}$,
and $B'$. We obtain a sequence of vertices $b_0b_1 b_2 b_3 \cdots b_{n_2+2} b_{n_2+3}
(=b_{2}') b_{n_2+4} ( =b_1') b_{n_2+5} ( =b_0')$, where $\{b_i,b_{i+1},b_{i+2}\}$
is $(4\e_5,\e_5^{3/4},\e_5^{1/2})$-typical for $ 0\leq i\leq n_2+3$ and $B'=\{b_3,b_4,\ldots,b_{n_2+2}\}$.
Consider a bipartite graph $\Gamma_B$ between $ A'$ and $B_1$, where $B_1=\{b_2, b_5, \dots, b_{3p_1-1}\}$
with $p_1=\lfloor \frac{n_2+3}{3}\rfloor$. For any $a\in A'$ and $b_i\in B_1$, $ab_i\in E(\Gamma_B)$ if and only if
$$a\in N(b_{i-2},b_{i-1},b_{i})\cap N(b_{i-1},b_{i},b_{i+1})\cap N(b_{i},b_{i+1},b_{i+2})\cap
N(b_{i+1},b_{i+2},b_{i+3}).$$
By the typicality of all triples, $d_{\Gamma_B}(b)\geq 0.99 n_1$ for any $b\in B_1$, and there are a lot of
vertices in $A'$ having large degree in $\Gamma_B$. We partition $A'=A_{big} \cup A_{small}$, where
$A_{big}=\{a\in A': d_{\Gamma_B} (a)\geq 0.9 p_1\}$. Then $|A_{small}|\leq 0.1n_1<3(m-1)$.
By Lemma~\ref{seq} again, there exists a sequence of vertices $ a_0 a_1 a_2 a_3
\dots a_{3m-2} a_{3m-1} a_{3m} (=a_2') a_{3m+1} (=a_1') a_{3m+2} (=a_0')$, such that$\{a_i, a_{i+1},
a_{i+2}\}$ is $\e_5^{1/2}$-typical for $0\le i\le 3m$ and $A_{small}\subseteq
\{a_3, a_4, \dots, a_{3m-1}\}$.

Consider a bipartite graph $\Gamma_A$ with partition classes $A_2, B''$, where $A_2=\{a_2, a_5, \dots,
a_{3m-1}\}$ and $B''$ is the set
containing the last $p_2=\lceil 0.3 n_1\rceil$ vertices of $B'$, i.e., $B''=\{b_{n_2+2},
b_{n_2+1}, \dots,$ $b_{n_2+3-p_2}\}$.  Then we can find a perfect matching in $\Gamma_A[A_2, \overline{B}]$,
where $\overline{B}$ is a subset of $B''$ and $|\overline{B}|=|A_2|$. Therefore, $$P_{top}=a_0 a_1 a_2
\overline{b_1} a_3 a_4 a_5 \overline {b_2} a_6 \cdots a_{3m-1} \overline{b_m} a_{3m} (=a_2')
a_{3m+1} (=a_1') a_{3m+2} (=a_0').$$
For the remaining vertices, we use a similar step in the proof of Lemma \ref{le3} to find $P_{zig}$,
and $Q \cup M_2\cup P_{top} \cup P_{zig}$
is a Hamiltonian cycle in $\h$.
\end{proof}

Since Lemma~\ref{le4} has some requirements on the number of vertices in $A\setminus V(Q \cup M_2)$ and
$B\setminus V(Q \cup M_2)$, the above proof works only if $m\in \mathbb{N}$. When $m$ is not an integer,
we use a ${\it good \ set}$ defined belows.
Here, we will consider the order of sets in $V_1V_2V_3$. For example, $AAB$ and $ABA$ are different.

\begin{definition}
For a 4-graph $\h$, let $A$, $B$ be a partition of $V(\h)$.
\begin{itemize}
\item The difference of a path $P$ in $\h$ with respect to $A, B$ is the number $p^*\equiv 3|V(P)\cap A|
-|V(P)\cap B|$ (mod 8).
\item An $(\e_1,\e_2,\e_3)$-switcher is a path $S$, which contains no $\e_5$-anarchists, has two
$(\e_1,\e_2,\e_3)$-typical end triples type of $BAA$ and $AAA$ or type $ABB$ and $BBB$, and has nonzero difference.

\item A set $X$ of vertices in $V(\h)$ is called good if $|X|<1600$ and $X$ does not contain any
$\e_5$-anarchists of $\h$, and, for any number $a\in \{0,1,2,\dots, 7\}$, there exists two disjoint bridges
$M_1$ and $M_2$ such that $V(M_1), V(M_2) \subseteq X$ and $m_1^*+m_2^*\equiv a$ (mod 8) where $m_i^*$ is
the difference of $M_i$ for $i=1,2$.
\end{itemize}
\end{definition}

{\bf Note.}
Given two disjoint $(\e_1,\e_2,\e_3)$-bridges $R_1$ and $R_2$, let $r_i^*$ be the difference of $R_i$ for
$i=1,2$. By Claim~\ref{triconn}, we can connected any two $(\e_1,\e_2,\e_3)$-bridges to obtain a path with both end triples in $AAA$ or in $BBB$. %any two $(\e_1,\e_2,\e_3)$-bridges can be connected by some $\e_5$-typical vertices.
For example, the $BBB$ triples of the two bridges can be connected by two vertices in $A$ and three vertices in $B$ (these vertices are from the vertex set $T$ by Claim~\ref{triconn}), and we have a path with both end triples in $AAA$. Adding some vertex from $B$ to one end of this path to make it have an $(\e_1,\e_2,\e_3)$-typical end triple of type $BAA$, we can obtain a path $P$ with difference $r_1^*+r_2^*+2$ and both end triples are $(\e_1,\e_2,\e_3)$-typical. Therefore, if
$r_1^*+r_2^*\not\equiv 6 \ (mod \ 8)$, $R_1$ and $R_2$ can form an $(\e_1,\e_2,\e_3)$-switcher with difference $r_1^*+r_2^*+2$.

\medskip

If there exists a good set $X$ in $\h$, then firstly, make a small modification of the partition of $V(\h)$ by
transferring all $\e_5$-anarchists to the other side, denoted by $(A',B')$.
Since $X$ is good, we can find two disjoint bridges $M_1$ and $M_2$ in $X$ to make $\frac{3 n_1'-n_2'+6}{8}$ an integer,
where $n_1'=|A'\setminus V(M_1\cup M_2)|$ and $n_2'=|B'\setminus V(M_1\cup M_2)|$ (We can do it since $X$ does not contain any $\e_5$-anarchists of $\h$). Next, by the proof of Lemma~\ref{le2}, there is a path $Q_{top}$
connecting all $\e_5$-medium vertices in $A'\setminus V(M_1\cup M_2)$ and a path $Q_{zig}$ connecting all $\e_5$-medium
vertices in $B'\setminus V(M_1\cup M_2)$. By Claim~\ref{triconn},
using $M_1$ connects $Q_{top}$ and $Q_{zig}$ to get a path $Q$ whose end triples are $AAA$ and $BBB$,
and both of them are $(\e_1,\e_2,\e_3)$-typical. It can be checked that $\frac{3 n_1-n_2+6}{8}$
is an integer when $\frac{3 n_1'-n_2'+6}{8}$ an integer, where $n_1=|A'\setminus V(Q\cup M_2)|$ and $n_2=|B'\setminus V(Q\cup M_2)|.$ (This is because $Q-M_1$ contains two paths $Q_1$ and $Q_2$ such that all edges of $Q_1$ are $AAAB$ edges and $Q_1$ contains $3x$ vertices in $A'$ and $x$ vertices in $B'$ while $Q_2$ of which all edges are $ABBB$ edges contains $3y$ vertices in $B'$ and $y$ vertices in $A'$ for some integers $x$ and $y$.)
Finally, by Lemma~\ref{le4}, we can find a Hamiltonian cycle in $\h$.

So the key is to prove the following lemma.

\begin{lemma}\label{le5}
For any 4-graph $\h$, let $A$, $B$ be a partition of $V(\h)$ and assume that \eqref{eq1},
\eqref{eq2}, \eqref{eq3} hold. Then $\h$ contains a good set $X$.
\end{lemma}

Next, we introduce a special type of edges, called {\it seed}. It is used to find a good set $X$.
\begin{definition}
A quadruple of vertices $(a,a',b,w)$ is called a seed if
\begin{itemize}
\item $a a' b w\in E(\h)$,
\item $\{a,a',b\}$ is $(\e_1,\e_2,\e_3)$-typical, and
\item $w\in B$ is $\e_5$-typical.
\end{itemize}
Similarly, a quadruple of vertices $(b,b',a,w)$ is called a seed if
\begin{itemize}
\item $b b' a w\in E(\h)$,
\item $\{b,b',a\}$ is $(\e_1,\e_2,\e_3)$-typical, and
\item $w\in A$ is $\e_5$-typical.
\end{itemize}
\end{definition}

\begin{claim}\label{switcher} Let $K\subseteq V(\h)$ with $|K|\le \e n$.
Given two disjoint seeds not intersecting with $K$,  $(a_i,a'_i,b_i,w_i)$, $i=1,2$, we can build a $K$-avoiding $(\e_1,\e_2,\e_3)$-switcher
of odd difference and at most $100$ vertices. Analogically, two disjoint seeds not intersecting with $K$,  $(b_i,b'_i,a_i,w_i)$, $i=1,2$, give a $K$-avoiding $(\e_1,\e_2,\e_3)$-switcher of odd difference and at most  $100$ vertices.
\end{claim}

\begin{proof} For the simplicity of the proof, we do not involve $K$ in our proof. But all vertices we need to choose in the following paragraphs can be chosen from the vertex set not intersecting with $K$ as the size of $K$ is small.

Since $w_1$ is $\e_5$-typical and $\delta_3(\h)\geq n-1$, we can extend the
edge $a_1a_1'b_1w_1$ to a path $P=a_1 a'_1 b_1 w_1 u_1 v_1$, such that $\{u_1,v_1\}$ is
an $(\e_1,\e_2)$-typical pair, $\{w_1,u_1,v_1\}$ is an $\e_5^{1/2}$-typical triple and $\{b_1,u_1,v_1\}$
is $(\e_1,\e_2,\e_3)$-typical. If $u_1v_1\in AA$, then there exist three $\e_1$-typical vertices
$\overline{a_1},a_1'',b_1'$ such that all triples $\{\overline{a_1},a_1,a'_1\}$, $\{u_1,v_1,a_1''\}$ and
$\{v_1,a_1'',b'_1\}$ are $(\e_1,\e_2,\e_3)$-typical. Hence, $S=\overline{a_1} a_1 a'_1 b_1 w_1 u_1 v_1 a_1''
b_1'$ is a path with both ends in $AAA$ and $AAB$, and then $S$ is a switcher with $s^*\equiv7$.

Otherwise, $u_1 v_1\in BA$, or $u_1 v_1\in AB$, or $u_1 v_1\in BB$. Similarly, extending the seed
$a_1 a'_1 b_1 w_1$ to a path $P=a_1 a'_1 b_1 w_1 u_1 v_1$, we can find
an $(\e_1,\e_2,\e_3)$-bridge $R_1, R_2\ or\ R_3$  with $r_1^*\equiv 6$, $r_2^*\equiv 7$ or $r_3^*\equiv 5$,
respectively.
We repeat the same construction on the second seed to get $P'=a_2 a'_2 b_2 w_2 u_2 v_2$. If we cannot
get a switcher with odd difference, then there exists an $(\e_1,\e_2,\e_3)$-bridge obtained from the second seed,
$R'_1$, or $R'_2$, or $R'_3$ with $(r'_1)^*\equiv 6$, or $(r'_2)^*\equiv 7$, or $(r'_3)^*\equiv 5$, respectively.
By applying Claim~\ref{triconn}, we can  use these two $(\e_1,\e_2,\e_3)$-bridges
$R_i$ and $R'_j$, $i,j\in\{1,2,3\}$, to form a switcher with odd difference.

\medskip

If the differences of these two bridges have different parity, connecting these two bridges results in a
switcher with odd difference. Now assume the difference of these two bridges have the same parity.

Suppose there are two bridges $R_1, R_2$ with even difference $r_1^*=r_2^*\equiv 6$. Then $u_1v_1
\in BA$ and $u_2v_2\in BA$. If $l_{a_1a_1'w_1}^B\geq \frac{n}{2}$, then we have a path $b_1' a_1
a_1' w_1  b_1 u_1 v_1$, where $b_1'\in B$ and $\{b_1',a_1,a_1'\}$ is $(\e_1,\e_2,\e_3)$-typical;
this path can be extended to a bridge $R$ with difference $r^*\equiv 3$.
% which results in a switcher with odd difference.
If $l_{a_1a_1'w_1}^A\geq \frac{n}{2}$, consider the path
$b_1 a_1 a_1' w_1 a_1''$ where $a_1''\in A$ is $\e_1$-typical. Extending this path,  we can get either
a switcher with difference 5, or a bridge with difference 5.

Now assume there are two bridges $R_1$ and $R_2$ with odd differences. Then $u_iv_i\in AB \ or
\ BB$ for $i=1,2$. If $l_{a_1'b_1w_1}^A\geq \frac{n}{2}$, we may assume $u_1v_1\in AB$; otherwise we obtain a
switcher with odd difference. If $l_{a_1a_1'w_1}^B\geq \frac{n}{3}$, then $a_1 a_1' w_1 b_1 u_1 v_1$
can be extended to a bridge with difference 4. Connecting this bridge with $R_2$ gives a switcher
with odd difference. We may assume $ l_{a_1a_1'w_1}^A\geq \frac{2n}{3}$. Then $|L_{a_1a_1'w_1}^A \cap L_{a_1' b_1 w_1}^A|\geq
\frac{n}{6}$. In this case, there exists $a\in A$ such that $a\in N(a_1,a_1',w_1)\cap N(a_1',b_1,w_1)$
and the pair $\{a,w_1\}$ is $\e_5^{3/4}$-typical. If $l_{a_1'w_1a}^A\geq \frac{n}{2}$, we have a
switcher with difference $5$ by the path $b_1 a_1 a_1' w_1 a$. If
$l_{a_1'w_1 a}^B\geq \frac{n}{2}$, the path $a_1 b_1 a_1' w_1 a$ gives a bridge with even
difference 2, which also gives a switcher with odd difference by connecting it and $R_2$.

If $l_{a_1'b_1w_1}^B\geq \frac{n}{2}$, we may assume $u_1v_1\in BB$ for all possible choices of
$u_1v_1$; otherwise we can obtain a bridge with even difference. Consider the path
$P_1=a_1 b_1 a_1' w_1 u_1$.  If $l_{a_1'w_1u_1}^B \geq \frac{n}{6}$, we extend $P_1$ to a bridge
with difference 0. We may assume $l_{a_1'w_1u_1}^A \geq \frac{5n}{6}$ and hence
$l_{a_1'w_1}^{AB}\geq \frac{n^2}{3}$, as there are at least $\frac{2n}{5}$ possible choices of
$u_1\in B$. Let $F:= \{ ab\in L_{a_1'w_1}^{AB}: \{a_1', a, b\} \text { is  $(\e_1,\e_2,\e_3)$-typical}\}$.
Since $a_1'$ is $\e_1$-typical and $l_{a_1' w_1}^{AB}\geq \frac{n^2}{3}$, $|F|\geq \frac{n^2}{4}$.
We know that $(a_1',a,b,w_1)$ is a seed for any $ab\in F$. Then for all possible $a\in A$ with
$ab\in F$ (the number of such vertices is at least $\frac{n}{4}$), we may assume
$l_{aw_1}^{AB}\geq \frac{n^2}{3}$. Otherwise we have a switcher with odd difference by the above
analysis. Hence, $l_{w_1}^{AAB}\geq \frac{n^2}{3} \cdot \frac{n}{4}\cdot \frac{1}{2} = \frac{n^3}{24}$,
contradicting the fact that $w_1$ is $\e_5$-typical.
\end{proof}

We know that connecting a bridge with a switcher forms a new bridge with different difference. If there are two given disjoint bridges with small lengths, switchers can help construct a good set. By Claim~\ref{switcher}, a lot of pairwise disjoint seeds give many switchers with odd differences. In the proof of Lemma~\ref{le5}, we explore when $\h$ contains many seeds or not and this completes the proof of Theorem~\ref{F}.

\begin{proof} [Proof of Lemma~\ref{le5} ] First, we claim that for sufficiently small $\e>0$, $\h$
contains two disjoint $(\e_1,\e_2,\e_3)$-bridges $M_1$ and $M_2$ with $|V(M_i)|\leq 25$ for $i=1,2$.

We repeat the proof of Lemma~\ref{le1} to build the first $(\e_1,\e_2,\e_3)$-bridge $M_1$, and
find two $(\e_1,\e_2)$-typical pairs $\{a_1,a_2\}$ and $\{b_1,b_2\}$. If $a_1a_2b_1b_2\in E(\h)$,
we can extend this edge to a bridge. If $a_1a_2b_1b_2\notin E(\h)$, then there exists a vertex
$z\in N(a_1,a_2,b_1)\cap N(a_1,b_1,b_2)$ such that we can extend $a_2 a_1 z b_1 b_2$ to an
$(\e_1,\e_2,\e_3)$-typical bridge.

To build the second bridge, find two $(\e_1, \e_2)$-typical pairs $\{a_1', a_2'\}$ and
$\{b_1', b_2'\}$, such that if $a_1'a_2'b_1'b_2'\notin E(\h)$ then
$a_1, a_2, b_1,b_2$ are not contained in the common neighbors of $\{a_1',a_2', b_1'\}$ and
$\{a_1',b_1',b_2'\}$, in order to get disjoint bridges. We can do it
since these vertices $a_1,a_2,b_1$ and $b_2$ are $\e_1$-typical, by Claim \ref{claim3},
there exists an $(\e_1,\e_2)$-typical pair $\{a'_1,b'_1\}$, such that all triples
$\{a'_1,b'_1,a_1\}$, $\{a'_1,b'_1,a_2\}$, $\{a'_1,b'_1,b_1\}$ and $\{a'_1,b'_1,b_2\}$ are
$\e_3$-typical. Then we know $d_A(a_1',b_1',b_i)\leq \e_3 n$ and
$d_B( a_1',b_1', a_i)\leq \e_3 n$ for $i=1,2$.
By Claim \ref{claim2}, we find two vertices $a'_2$, $b'_2$, such that $\{a'_1, a'_2\}$ and
$\{b'_1,b'_2\}$ are $(\e_1,\e_2)$-typical pairs and for $i=1,2$,
$a'_2 a'_1 b'_1 b_i\notin E(\h)$, and $b'_2 b'_1 a'_1 a_i\notin E(\h)$. Similarly,
if $a'_2 a'_1 b'_1 b'_2\in E(\h)$, $b'_2 b'_1 a'_1 a'_2$ can be extended to an
$(\e_1,\e_2,\e_3)$-typical bridge $M_2$   and $M_1\cap M_2=\emptyset$, since
$\delta(\h)\geq n-1$. Otherwise, there exist two vertices
$z',z'_1$ different from $a_1,a_2,b_1,b_2$ satisfying $z',z'_1\in N(a'_2,a'_1,b'_1)\cap
N(b'_2,b'_1,a'_1)$. Without loss of generality, suppose $z'\neq z$, then $a'_2 a'_1 z' b'_1
b'_2$ can be extended to a $(\e_1,\e_2,\e_3)$-bridge $M_2$ such that $M_1\cap M_2=\emptyset$.
So in any case, there are two disjoint $(\e_1,\e_2,\e_3)$-bridges $M_1$ and $M_2$ in
$\h$ with $|V(M_i)|\leq 25$ for $i=1,2$.
\begin{itemize}
\item []  {\bf Case 1.} All vertices in $B$ are $\e_5$-typical (or all vertices in $A$ are $\e_5$-typical).
\end{itemize}
We may consider the case when all vertices in $B$ are $\e_5$-typical.
Let $V':=V(M_1\cup M_2)$. It suffices to show that $\h$ has fourteen pairwise disjoint seeds of type $(a,a',b,w)$ that are
also disjoint from $V'$. Then by Claim~\ref{switcher}, every two such seeds can form a switcher with odd difference.
Hence, we can obtain seven pairwise disjoint switchers and each has odd difference.

Since all $b\in B$ are $\e_5$-typical, we have $l_b^{AAB}\leq \e_5 n^3$. Consider the set of
triples $E=\bigcup_{b\in V'\cap B}L_b^{AAB}$.
Since $|V'|\leq 50$, we have
$|E|\leq 50\e_5 n^3$ and, thus, by Corollary \ref{coro1}, there exists an
$(\e_1,\e_2,\e_3)$-typical triple $\{a_1,a'_1,b_1\}$ such that $a_1a_1'b_1\notin E$ and
$a_1,a'_1,b_1\notin V'$. Let $w_1\in
N_B(a_1,a'_1,b_1)$; the existence of $w_1$ follows from \eqref{eq1}.
By the definition of $E$, $w_1\notin V'$. So we get a seed $(a_1,a'_1,b_1,w_1)$.
Assume that we have produced $i-1$ seeds, $(a_j,a'_j,b_j,w_j)$, for $j=1,\ldots,i-1$. Set
\[
E_{i-1}=E\cup (L_{b_1}^{AAB}\cup L_{w_1}^{AAB})\cdots \cup (L_{b_{i-1}}^{AAB}\cup L_{w_{i-1}}^{AAB})
\]
and note that $|E_{i-1}|\leq 100\e_5 n^3$ if $i\le 15$. Similarly as before, we can find an
$(\e_1,\e_2,\e_3)$-typical triple $\{a_i,a'_i,b'_i\}$ such that $a_ia'_ib_i\notin E_{i-1}$ and
$a_i,a'_i,b_i\notin V'\cup \{a_1,a'_1,b_1,w_1\}\cup \cdots \cup \{a_{i-1},a'_{i-1},b_{i-1},w_{i-1}\}$.
We can also find $w_i\in N_B(a_i,a'_i,b_i)$ such that $a_i,a'_i,b'_i,w_i$ is a seed.
So there are at least fourteen pairwise disjoint seeds and, applying Claim~\ref{switcher},
we can form seven pairwise disjoint $(\e_1,\e_2,\e_3)$-switchers with odd differences. For any $a\in \{0,1,2,\ldots,7\}$,
we can find some numbers from those seven odd differences such that the summation of them is $a$ (mod 8). In particular, for the case $a=0$, we do not use any switchers.
Let $V''$ denote the set of all vertices of these seven switchers.
%We know that connecting a bridge with a switcher forms a new bridge with different difference.
%Since we have seven disjoint switchers with odd differences,
Then $V'\cup V''$ and a small number of $\e_5$-typical vertices, which are used to
connect bridges and switchers, form  a good set in $\h$.
\begin{itemize}
\item[] {\bf Case 2.} There exists an $\e_5$-anarchist in $\h$.
\end{itemize}
By Claim~\ref{fact1}, all vertices in one side are $3\e_5$-typical. Then a similar proof argument as
in Case 1 completes this case.
\begin{itemize}
\item[] {\bf Case 3.} There exists an $\e_5$-medium vertices in $\h$ and $\h$ doesn't contain any
$\e_5$-anarchist.
\end{itemize}
We may assume there are at least two $\e_5$-medium vertices, otherwise we get back to Case 1.
If there are at least $28$ pairwise disjoint seeds in $\h$ that are also disjoint from $V'=V(M_1\cup M_2)$, then there exist at least $14$ pairwise disjoint seeds of the same type. Then we can find a good set $X$ by Claim~\ref{switcher} and the argument in Case 1. So we may assume that the number of pairwise disjoint seeds is less than $28$.  Let $V_s$ denote a maximal set of vertices containing
pairwise disjoint seeds in $\h-V'$ and let $V_m$ denote the set of $\e_5$-typical vertices in $V'$.
Then all vertices in $V:=V_s\cup V_m$ are $\e_5$-typical. Let $V_A=V\cap A$ and $V_B=V\cap B$.
Then $|V_A|\leq 2\cdot 25 + 2\cdot 28=106$ and $|V_B|\leq 2\cdot 25 + 2\cdot 28=106$ (by $|M_i|\le 25$ for $i=1,2$ and the fact that each seed forms an $AABB$ edge).

Let $E_A=\bigcup_{a\in V_A}L_a^{ABB}$, $E_B=\bigcup_{b\in V_B}L_b^{AAB}$. Then $|E_A|\leq 106\e_5 n^3$ and
$|E_B|\leq 106\e_5 n^3$ since all vertices in $V$ are $\e_5$-typical. Let $T$ be a set of quadruples $(a_i,a_j,b_k,b_l)$ such that both $\{a_i,a_j,b_k\}$ and $\{a_j,b_k,b_l\}$ are $(\e_1,\e_2,\e_3)$-typical triples, $a_ia_jb_k\notin E_B$, $a_jb_kb_l\notin
E_A$ and $a_i,a_j,b_k,b_l\notin V$, where $a_i\neq a_j\in A$ and $b_k\neq b_l\in B$.
By Corollary~\ref{coro1}, $\h$ contains at most $\e_4n^3$ $(\e_1,\e_2,\e_3)$-atypical triples.
So $|T|\geq n^2(n-1)^2-2\cdot \e_4 n^3\cdot (n-1) - 2 \cdot 2 \cdot106\e_5 n^3\cdot
(n-1)-2\cdot 106 \cdot n(n-1)^2>\frac{n^4}{2}$.

For any $(a_i,a_j,b_k,b_l)\in T$, $a_i,a_j,b_k,b_l\notin V'\cup V_s$ (by the definition of $T$)
and $a_ia_jb_kb_l\notin \h$ (by the maximality of $V_s$). Since $a_ia_jb_kb_l\notin \h$, it follows from the proof of Lemma~\ref{le1} that $|N(a_j,b_k,b_l)\cap N(a_i,a_j,b_k)|\geq 2$.
Now we claim that for each vertex $v\in N(a_j,b_k,b_l)\cap N(a_i,a_j,b_k)$, either $v\in V'\cup V_s$ or $v$ is $\e_5$-medium. Suppose $v\notin  V'\cup V_s$ and $v$ is not $\e_5$-medium. Then $v$ is $\e_5$-typical, and hence, if $v\in B$ then $(a_i,a_j,b_k,v)$ is a seed disjoint from $V'\cup V_s$, and if $v\in A$ then $(b_k, b_l,a_j, v)$ is a seed disjoint from $V'\cup V_s$. This contradicts the maximality of $V_s$.
% All vertices of $N_A(a_j,b_k,b_l)$ and$N_B(a_i,a_j,b_k)$ are $\e_5$-medium. Using the similar way of
%constructing a bridge,
%we have $|N(a_j,b_k,b_l)\cap N(a_i,a_j,b_k)|\geq 2$. 
Since the number of $\e_5$-medium vertices in $\h$ is at most $8\e_0/\e_5n$ and $|V'\cup V_s|\le 2\cdot 25+ 28\cdot 4= 162$, the number of all possible vertices in $N(a_j,b_k,b_l)\cap N(a_i,a_j,b_k)$ for all $(a_i,a_j,b_k,b_l)\in T$ is at most  $8\e_0/\e_5n+162 <\frac{\e^3n}{10}$. Therefore,  we can find a vertex $u$, such that at least $\frac{n^4}{2}/ \frac{\e^3n}{10}=\frac{5n^3}{\e^3}$ quadruples $(a_i,a_j,b_k,b_l)$ satisfy $u\in N(a_i,a_j,b_k)\cap
N(a_j,b_k,b_l)$. Since $|N(a_j,b_k,b_l)\cap N(a_i,a_j,b_k)|\geq 2$, we can find two
such vertices $u,v$ by applying Pigeonhole principle twice.

\medskip

Next, we claim that for each integer $i\in I=\{0,3,6,7\}$, $u$ is contained in a bridge $U_i$ with difference $i$, and 
$v$ is contained in a bridge $V_i$ with difference $i$. Moreover,
$(\cup_{i\in I} V(U_i)) \cap (\cup_{i\in I} V(V_i))=\emptyset.$

Without loss of generality, we may assume $u \in B$. (The proof for the case when $u\in A$ is analogous.) Construct an auxiliary bipartite graph $G$ with partition classes $Y, Z$, where $Y=\{ (a_i,a_j) : a_i,a_j\in A, a_i\neq a_j \}$ and $Z=\{ (b_k,b_l) : b_k,b_l\in B,
b_k\neq b_l \}$, and $(a_i,a_j)\sim (b_k,b_l)$ if and only if $(a_i,a_j,b_k,b_l)\in T$
and $u\in N(a_i,a_j,b_k)\cap N(a_j,b_k,b_l)$. Then $|E(G)|\geq \frac{5n^3}{\e^3}>8n^3$
and the average degree of $G$ is at least $8n$.

Note that every graph $H$ contains a subgraph $D$, of which the minimum degree is at least
half of the average degree of $H$. Hence there exists $G'\subseteq G$ such that
$\delta(G')\geq 4n$.
In $G'$, $d_{G'} ((a_1,a_2))\geq 4n$ for $(a_1,a_2)\in V(G')\cap Y$. There exist $(b_1,b_3),
(b_2, b_3)\in V(G')\cap Z $ such that $(a_1,a_2)(b_1,b_3), (a_1,a_2)(b_2,b_3)\in E(G')$ with
$b_1\neq b_2$; otherwise $d_{G'} ((a_1,a_2))\leq n$. Since $d_{G'} ((b_2,b_3))\geq 4n$, there
exists $(a_3,a_4)\in V(G')\cap Y$ such that $(a_3,a_4)\in N_{G'}( (b_2,b_3))$ and $a_3,a_4\notin
\{ a_1,a_2\}$. Hence, we have $$u\in N(a_1,a_2,b_1)\cap N(a_2,b_1,b_3)\cap N(a_1,a_2,b_2)\cap
N(a_2,b_2,b_3)\cap N(a_3,a_4,b_2)\cap N(a_4,b_2,b_3).$$

Now the path $a_1 b_1 a_2 u b_3 b_2 a_4$ can be extended to a bridge with difference 0 by Claim
\ref{triconn}.
Similarly, the path $b_1 a_1 a_2 u b_2 b_3 a_4$ gives a bridge with difference 3
and the path $a_3 a_4 b_2 u b_3 a_2 b_1$ gives a bridge with difference 6.

To obtain a bridge with difference 7, we consider another bipartite graph $H$ with partition classes
$U, W$, where $U=\{ (a_i,b_k) : a_i\in A, b_k\in B \}$ and $W=\{ (a_j,b_l) : a_j\in A, b_l\in B \}$,
and $(a_i,b_k)\sim (a_j,b_l)$ if and only if $(a_i,a_j,b_k,b_l)\in T$ and $u\in N(a_i,a_j,b_k)\cap
N(a_j,b_k,b_l)$. Then $|E(H)|\geq \frac{5n^3}{\e^3}>8n^3$ and the average degree of $H$ is at least
$8n$. Similarly, for some $(a_1,b_1)\in U$, there exists $(a_2,b_2), (a_3,b_3)\in W$ such that
$a_2\neq a_3, d_{H}( (a_2,b_2))\geq 4n$ and $(a_1,b_1)(a_2,b_2), (a_1,b_1)(a_3,b_3)\in E(H)$.
Since $d_H((a_2,b_2))\geq 4n$, there exists $(a_4,b_4)\in U$ such that $(a_4,b_4)\in N_{H} ((a_2,b_2))$
and $b_4\neq b_1$. Hence, we have $$u\in N(a_1,a_2,b_1)\cap N(a_2,b_1,b_2)\cap N(a_1,a_3,b_1)
\cap N(a_2,b_4,b_2).$$ The path $a_3 a_1 b_1 u a_2 b_2 b_4$ results in a bridge with difference 7.

To summarize, we found four bridges with difference 0, 3, 6, 7 respectively, and all contain $u$.
Let $V_1'$ be the set of vertices of these four bridges. Since each bridge is obtained by extending a path with $7$ vertices and both end triples $(\e_1,\e_2,\e_3)$-typical by the application of Claim~\ref{triconn}, it has at most $7+2\cdot(12-3)=25$ vertices. Then $|V_1'|\leq 4\cdot 25= 100$.
Repeat the same argument for $v$, we find four bridges with difference 0, 3, 6, 7 respectively,
and all are disjoint from $V_1'$ and contain
$v$. We complete the proof of this claim.

\medskip
Now we find a good set $X$. Let $V_2'$ be the set of vertices of such four bridges containing $v$.
We can choose one bridge $M_1$ containing $u$ and one bridge $M_2$ containing $v$ to make
$\frac{3n_1'-n_2'+6}{8}$ an integer where $n_1'=|A\setminus V(M_1\cup M_2)|$ and
$n_2'=|B\setminus V(M_1\cup M_2)|$, since
\begin{align*}
& 0+0\equiv 0  (\text{mod } 8);\ 3+6\equiv 1  (\text{mod } 8);\ 3+7\equiv 2  (\text{mod } 8);\ 0+3\equiv 3  (\text{mod }8);\\&
 6+6\equiv 4 (\text{mod }8); \ 6+7\equiv 5  (\text{mod }8); \ 0+6\equiv 6 (\text{mod }8); \ 0+7\equiv 7 (\text{mod }8).
\end{align*}
Therefore, $X=V_1'\cup V_2'$ is a good set in $\h$.
\end{proof}

\section{Concluding remarks}
For the case when $|V(\h)|=2n+1$,  choose a partition $A$, $B$ of $V(\h)$ with $|A|=n+1$ and $|B|=n$
and, subject to this, $|\h(A,A,B,B)|$ is minimal. The proof is almost exactly the same as that of Theorems
\ref{E} and \ref{F}, since one extra vertex almost does not make any difference here. Theorem \ref{main}
will become the following

\begin{theorem}\label{odd case}
There exists $\e_0>0$ such that, for sufficiently large $n$ and any 4-graph $\h$ on $2n+1$ vertices
with $b(\h)<\e_0 n^4$, the following hold.
\begin{enumerate}
\item[(i)]If $\delta_3(\h)\geq n-1$, then $\h$ has a Hamiltonian path;
\item[(ii)]If $\delta_3(\h)\geq n$, then $\h$ has a Hamiltonian cycle.
\end{enumerate}
\end{theorem}

Thus, if a 4-graph $\h$ with $n$ vertices is close to extremal graph $\h_0$ and its
minimum co-degree is at least $\lfloor \frac{n-1}{2}\rfloor$, then $\h$ must contain a Hamiltonian
cycle. It remains to consider the other case, that is, when $\h$ is far from $\h_0$.

\begin{conjecture}\label{conj2}
For all $c>0$ there exists $c_1>0$ such that, for sufficiently large $n$ and a 4-graph on $n$ vertices,
if $b(\h)\geq cn^4$ and $\delta_3(\h)\geq (1-c_1)\frac{n}{2}$ then $\h$ has a Hamiltonian cycle.
\end{conjecture}

Conjecture \ref{conj2} is equivalent to Conjecture \ref{conj1} for $k=4$. It is likely that this case
requires the use of absorption technique that R\"{o}dl, Ruci\'{n}ski and
Szemer\'{e}di \cite{Rodl2011} used to prove the case of 3-graphs.

On the other hand, using the tools in this paper, one might ask if Theorem \ref{main} holds for
$k$-graphs with $k\ge 5$.
\begin{conjecture}
There exists $\e_0>0$ such that, for sufficiently large $n$ and any k-graph $\h$ on $n$ vertices with
$b(\h)<\e_0 n^k$ the following holds:
If $\delta_{k-1}(\h)\geq  \lfloor\frac{n-k+3}{2}\rfloor$, then $\h$ has a Hamiltonian cycle.
\end{conjecture}

\section*{Acknowledgements}
We are thankful to the anonymous referees for their valuable comments and suggestions.


\begin{thebibliography}{99}

\addtolength{\baselineskip}{-1ex}

\bibitem{Bermond1976} J. C. Bermond, A. Germa and M. C. Heydemann, Hypergraphes hamiltoniens, {\it Prob. Comb. Theorie Graph
Orsay} {\bf 260} (1976), 39--43.

\bibitem{Czygrinow2014} A. Czygrinow and T. Molla, Tight codegree condition for the existence of loose Hamilton cycles in
3-graphs, {\it SIAM J. Discrete Math.} {\bf 28} (2014), 67--76.

\bibitem{Dirac1952} G.A. Dirac, Some theorems of abstract graphs, {\it Proc. Lond. Math. Soc. (3)} (1952) 69--81.

\bibitem{Frankl2002} P. Frankl and V. R\"{o}dl, Extremal problems on set systems, {\it Random Structures Algorithms} {\bf 20} (2002) 131--164.

\bibitem{Gy2008} A. Gy\'{a}rf\'{a}s, J. Lehel, G. S\'{a}rk\"{o}zy and R. Schelp, Monochromatic Hamiltonian Berge-cycles in colored complete uniform
hypergraphs, {\it J. Combin. Theory Ser. B} {\bf 98} (2008) 342--358.

\bibitem{Han2015} J. Han and Y. Zhao, Minimum vertex degree threshold for loose Hamilton cycles in 3-uniform
hypergraphs, {\it J. Combin. Theory Ser. B} {\bf 114} (2015), 70--96.

\bibitem{Han2010} H. H\`{a}n and M. Schacht, Dirac-type results for loose Hamilton cycles in uniform hypergraphs, {\it J. Combin. Theory Ser.
B} {\bf 100} (2010) 332--346.

\bibitem{Haxell2006} P. Haxell, T. Luczak, Y. Peng, V. R\"{o}dl and A. Ruc\'{i}nski, M. Simonovits and
J. Skokan, The Ramsey number for hypergraph cycles I, {\it J. Combin. Theory Series A}
{\bf 113} (2006), 67--83.

\bibitem{Janson2000} S. Janson, T. Luczak and A. Ruc\'{i}nski, Random Graphs, {\it John Wiley and Sons},
New York (2000).

\bibitem{Katona1999} Gy. Y. Katona and H. A. Kierstead, Hamiltonian chains in hypergraphs, {\it J. Graph
Theory} {\bf 30} (1999), 205--212.

\bibitem{Kuhn2006} D. K\"{u}hn and D. Osthus, Loose Hamilton cycles in 3-uniform hypergraphs of large
minimum degree, {\it J. Combin. Theory Series} B {\bf 96} (2006), 767--821.

\bibitem{Reiher2019} C. Reiher, V. R\"{o}dl, A. Ruci\'{n}ski, M. Schacht and E. Szemer\'{e}di,
Minimum vertex degree condition for tight Hamiltonian cycles in 3-uniform hypergraphs,
{\it Proc. Lond. Math. Soc. (3)} {\bf 119} (2019), 409--439.

\bibitem{Rodl 2006} V. R\"{o}dl, A. Ruci\'{n}ski and E. Szemer\'{e}di, A Dirac-type theorem for 3-uniform hypergraphs, {\it Combin. Probab. Comput.}
{\bf 15} (2006) 229--251.

\bibitem{Rodl2008} V. R\"{o}dl and A. Ruci\'{n}ski, E. Szemer\'{e}di, An approximate Dirac-type theorem for k-uniform hypergraphs, {\it Combinatorica}
{\bf 28} (2008) 229--260.

\bibitem{Rodl2011} V. R\"{o}dl, A. Ruci\'{n}ski and E Szemer\'{e}di,
Dirac-type conditions for Hamiltonian paths and cycles in 3-uniform hypergraphs,
{\it Adv. Math.} {\bf 227} (2011), 1225--1299.

\bibitem{Rodl2017} V. R\"{o}dl, A. Ruci\'{n}ski, M. Schacht and E. Szemer šŠdi,  On the Hamiltonicity of triple systems with high minimum degree.
{\it Ann. Comb.} {\bf 21} (2017), 95--117.

\end{thebibliography}
\end{document}